\documentclass{amsart}

\usepackage{graphicx}
\usepackage{url}
\usepackage{amsmath}
\usepackage{amsthm}
\usepackage{amsfonts}
\usepackage[all]{xy}
\usepackage{multicol}
\usepackage{hyperref}
\usepackage{geometry}

\usepackage[percent]{overpic}
\usepackage{subfigure}

\newcommand{\ZZ}{\mathbb {Z}}
\newcommand{\RR}{\mathbb {R}}
\newcommand{\CC}{\mathbb {C}}
\newcommand{\NN}{\mathbb {N}}

\newtheorem{thm}{Theorem}[section]
\newtheorem{prop}[thm]{Proposition}
\newtheorem{cor}[thm]{Corollary}
\newtheorem{lemma}[thm]{Lemma}
\newtheorem{notation}[thm]{Notation}
\newtheorem{defn}[thm]{Definition}
\newtheorem{rem}[thm]{Remark}
\newtheorem{ex}[thm]{Example}
\numberwithin{equation}{section}

\title{On the Alexander polynomial of links in lens spaces}

\author[E. Horvat, B. Gabrov\v sek]{}

\date{\today}


\begin{document}

\maketitle
\medskip	
\centerline{\scshape
Eva Horvat}
\medskip
{\footnotesize
	\centerline{University of Ljubljana, Faculty of Education, SI-1000 Ljubljana, Slovenia}
	\centerline{eva.horvat@pef.uni-lj.si}
	}
\medskip	
\centerline{\scshape
Bo\v stjan Gabrov\v sek}
\medskip
{\footnotesize
	\centerline{University of Ljubljana, Faculty of Mechanical Engineering, SI-1000 Ljubljana, Slovenia}
	\centerline{University of Ljubljana, Faculty of Mathematics and Physics, SI-1000 Ljubljana, Slovenia}
	\centerline{bostjan.gabrovsek@fs.uni-lj.si}
	}
\bigskip

\begin{abstract}
We show how the Alexander polynomial of links in lens spaces is related to the classical Alexander polynomial of a link in the 3-sphere, obtained by cutting out the exceptional lens space fibre. It follows from these relationship that a certain normalization of the Alexander polynomial satisfies a skein relation in lens spaces. 
\end{abstract}


\maketitle

\begin{section}{Introduction \& Background}\label{sec-intro}

In this paper we are interested in the Alexander polynomial, which is perhaps one of the most extensively studied invariants in classical knot theory. We now know that the Alexander polynomial holds information about sliceness, fiberdness, and knot symmetries.

The first definition of the Alexander polynomial was constructed from the homology of the infinite cyclic cover of knot complement in 1928 by Alexander~\cite{alex}, but it soon became clear that the polynomial can be studied from several different viewpoints. 
In 1962 Milnor observed that the polynomial can be defined through Reidemeister torsion~\cite{milnor}, 
a few years later Conway showed that a certain normalization of the polynomial is characterised using local skein relations, although this approach was only popularized due to Kauffman in the 80s~\cite{kauff}. Another notable construction, closely related to Alexander's original idea, arises from the fundamental group of the knot's complement via Fox calculus~\cite{fox}. 

In 1975 Turaev extended Milnor's idea and defined the Alexander polynomial for links in other 3-manifolds~\cite{turaev}.
In 1990 Lin presented the idea of a twisted Alexander polynomial which generalizes the classical Alexander polynomial~\cite{lin}.

While various definitions of the Alexander polynomial for links in $S^3$ coincide, it was not until 2008 that Huynh and Le showed that a normalized version of the Alexander polynomial satisfies the skein relation for links in the projective space~\cite{huynh}. The question whether the Alexander polynomial satisfies a skein relation in other 3-manifolds naturally arises.


In \cite{torres}, Torres showed the renowned Torres-Fox formula that relates the multivariable Alexander polynomial of an $r$-component link $L=K_1\cup K_2\cup \cdots \cup K_r$ to the multivariable Alexander polnomial of the link $L$ with component $K_r$ removed:


$$\Delta(t_1,\ldots,t_{r-1},1) = \begin{cases}
\Delta(t_1) (t_1^{l_1}-1)/(t_1-1), & \mbox{if } r=2, \\
\Delta(t_1,\ldots,t_{r-1}) (t_1^{l_1} t_2^{l_2}\cdots t_{r-1}^{l_{r-1}}-1), & \mbox{if } r>2, \\

\end{cases}$$

where $l_i = lk(K_i,K_r)$ is the linking number between the components $K_1$ and $K_r$. A Torres-type formula for the twisted Alexander polynomial was given by Morifuji in~\cite{mor}.

In this paper we show how the Alexander polynomial of a link with one trivial (unknotted) component is related to the Alexander polynomial of the link if we perform $-p/q$ surgery on the trivial component (the latter link can be viewed as a link in the lens space $L(p,q)$).
Using this relation, we show that the normalized version of the Alexander polynomial of links in lens spaces and other 3-manifolds indeed satisfies a skein relation. 



The paper is organized as follows. In Section~\ref{sec-fund}, we give an explicit construction for the presentation of the fundamental group of the complement of a link in a lens space. In Section~\ref{sec-fox}, we present a definition of the (twisted) Alexander polynomial via the Alexander-Fox matrix.
In Section~\ref{sec-main}, the first main result is given (Theorem~\ref{th4}), namely the connection between the Alexander polynomial of a link in a lens space and of its classical counterpart in $S^3$, derived from the connection between a link in $S^3$ and the link obtained by performing a rational Dehn surgery on a trivial component. Finally, in Section~\ref{sec:skein} we show that a normalization of the Alexander polynomial for links in lens spaces satisfies a skein relation (Theorem~\ref{thm:skein1}). We also study a generalization of our results for links in other closed, orientable, connected 3-manifolds (Theorem~\ref{th5}). 

\end{section}

\begin{section}{Presentation of the group of a link in a 3-manifold}\label{sec-fund}
The Alexander polynomial of a link may be derived from a presentation of the link group using Fox calculus~\cite{fox}. A classical link $L$ in the 3-sphere $S^{3}$ has a widely known Wirtinger presentation for  $\pi _{1}(S^{3}\backslash L,*)$, which can be generalized by producing a presentation of the link group for links in the lens spaces $L(p,q)$. Finally, we describe a presentation of the group of a link in any closed, connected, orientable 3-manifold $M$.  

\begin{subsection}{The group of a link in $S^{3}$}
\label{sub11}
We begin with the classical case, where $L$ is a link in $S^{3}$. Choose a basepoint $*\in S^{3}\backslash L$. A diagram of $L$ is a finite collection of arcs $\alpha _{1},\alpha _{2},\ldots ,\alpha _{n}$ that meet at crossings. To each arc $\alpha _{i}$, we associate the homotopy class $x_{i}$ of a simple loop based at $*$ that links $\alpha _{i}$ with linking number 1, while not linking any other arc $\alpha _{j}$ for $j\neq i$. Imagine three arcs $\alpha _{i_{1}},\alpha _{i_{2}}, \alpha _{i_{3}}$ that meet at a crossing of the diagram. The corresponding generators of $\pi _{1}(S^{3}\backslash L,*)$ are subject to the Wirtinger relation $x_{i_{1}}x_{i_{3}}x_{i_{2}}^{-1}x_{i_{3}}^{-1}=1$ if the corresponding crossing is positive, or $x_{i_{1}}x_{i_{3}}^{-1}x_{i_{2}}^{-1}x_{i_{3}}=1$ if the corresponding crossing is negative, see Figure~\ref{figwirt}. We obtain the following result: 
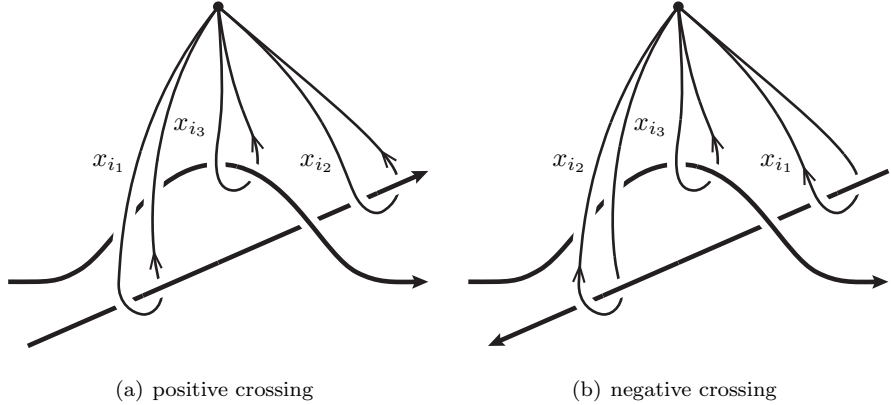
\begin{figure}[ht]
	\centering
	\subfigure[positive crossing]{\begin{overpic}[page=1]{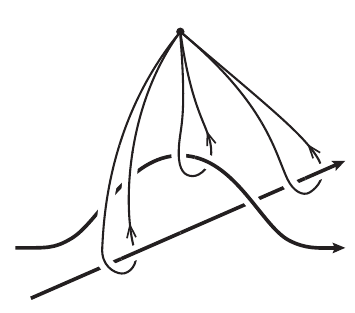}
		\put(23,44){$x_{i_1}$}\put(69,44){$x_{i_2}$}\put(41,52){$x_{i_3}$}
		\label{fig-wirt-1}
	\end{overpic}}\hspace{0ex}
	\subfigure[negative crossing]{\begin{overpic}[page=2]{images}
		\put(23,44){$x_{i_2}$}\put(69,44){$x_{i_1}$}\put(41,52){$x_{i_3}$}
		\label{fig-wirt-2}
	\end{overpic}}
	\caption{Wirtinger relations.}\label{figwirt}
\end{figure}

\begin{thm}[\cite{rolfsen}]
\label{th1}
Let $L\subset S^{3}$ be a link in the 3-sphere, given by a plane diagram with $n$ crossings. Using the introduced notation, the group of the link $L$ has a presentation $$\pi _{1}(S^{3}\backslash L,*)=\left <x_{1},\ldots ,x_{n}|\, r_{1},\ldots ,r_{n}\right >\;,$$ where $r_{i}$ for $i=1,\ldots ,n$ is the Wirtinger relation $x_{i_{1}}x_{i_{3}}x_{i_{2}}^{-1}x_{i_{3}}^{-1}$ (positive crossing) or $x_{i_{1}}x_{i_{3}}^{-1}x_{i_{2}}^{-1}x_{i_{3}}$ (negative crossing), corresponding to the $i$-th crossing of the link $L$. 
\end{thm}

By the abelianization of the link group, all the generators belonging to the same component become homologous, which implies the following result.

\begin{cor}[\cite{rolfsen}]
\label{cor1}
Let $L\subset S^{3}$ be a link with $r$ components. Then the first homology group of the link complement equals $H_{1}(S^{3}\backslash L)\cong \ZZ ^{r}$. 
\end{cor}

\end{subsection}
\begin{subsection}{The group of a link in a lens space}
\label{subs0}
Let $p$ and $q$ be coprime integers such that $0 < p < q$. The lens space $L(p,q)$ may be constructed as follows. Describe the 3-sphere $S^{3}$ as a union of two solid tori $V_{1}$ and $V_{2}$ (corresponding to the Heegaard decomposition of genus 1). Choose a meridian $m_{1}$ and a longitude $l_{1}$, generating $\pi _{1}(\partial V_{1})$. Then $L(p,q)$ is obtained from $S^{3}$ by $-p/q$ surgery on $V_{2}$, i.e., we remove $V_{2}$ from $S^{3}$ and then glue it back onto $V_{1}$ by the boundary homeomorphism $h\colon \partial V_{2}\to \partial V_{1}$ which maps the meridian $m_{2}$ of $V_{2}$ to the $(p,-q)$-curve on $\partial V_1$, $$h_{*}(m_{2})=pl_{1}-qm_{1},$$
where $h_*$ is the induced homology map, see Figure~\ref{fig-heeg}. 

\begin{figure}[ht]
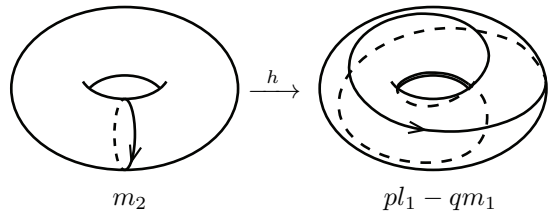

	\centering
	\begin{overpic}[page=4]{images}
		\put(45,-10){$m_2$}
	\end{overpic}\raisebox{1.1cm}{$\xrightarrow{\;\;h\;\;}$}
	\begin{overpic}[page=3]{images}
		\put(30,-10){$pl_1-qm_1$}
	\end{overpic}
	\vspace{1ex}
	\caption{Heegaard decomposition of $L(p,q)$.}\label{fig-heeg}
\end{figure}

The Kirby diagram of $L(p,q)$ is the unknot $U$ bearing the framing $-p/q$, see Figure~\ref{fig-kirby-1} (also see~\cite{kirby, mr1}).
The unknot represents the meridian $m_{1}$, on whose regular neighbourhood the surgery is performed. Thus, the lens space $L(p,q)$ is completely defined by three data: $p$, $q$ and the position of the meridian $m_{1}$.

\begin{figure}[ht]
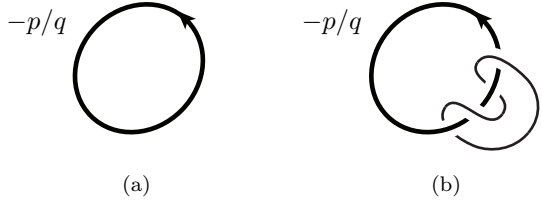

	\centering
	\subfigure[]{
	\begin{overpic}[page=14]{images}
		\put(-16,64){$-p/q$}
	\end{overpic}\label{fig-kirby-1}
	}\hspace{7ex}
	\subfigure[]{
	\begin{overpic}[page=13]{images}
		\put(-23,64){$-p/q$}
	\end{overpic}\label{fig-kirby-2}
	}
	
	\caption{The Kirby diagram for $L(p,q)$ and diagram of a knot in $L(p,q)$.}\label{fig-kirby}
\end{figure}

Now consider a link $L$ in $L(p,q)$,
for which may assume $L \subset V_{1}$. Represent $L(p,q)$ by its Kirby diagram as an unknot $U$ bearing the framing $-p/q$, and draw $L$ in relation with this diagram, see Figure~\ref{fig-kirby-2}. Such diagrams are also called mixed link diagrams in the language of~\cite{LR1, LR2, lambro, DLP}.
See also~\cite{CMR, BC, mr2, gma2} for alternative diagrams.

\begin{defn}
A link $L\subset L(p,q)$ is \textbf{affine} (sometimes also called local), if it is contained in a 3-ball $B^{3}\subset L(p,q)$.
\end{defn}


Let $\nu(L\cup U)$ and $\nu(L)$ be the open regular neighbourhoods of $L\cup U$ and $L$, respectively. The CW-decompositions of $S^3 \setminus \nu(L \cup U)$ and $L(p,q) \setminus \nu (L)$ differ by adding a 2-handle, the meridional disc bounded by the meridian $m_2$ of $V_2$ (Figure~\ref{fig-heeg}) and adding a 3-handle to close the component.
The extra 2-handle adds, in multiplicative notation, the relation $m_{1}^{p}\cdot l_{1}^{-q}$ to the fundamental group.
We obtain the following two well-known results (\cite{rolfsen}).

\begin{thm}[\cite{rolfsen}]
\label{th2}
Let $L$ be a link in the lens space $L(p,q)$. Represent $L(p,q)$ by an unknot $U$ bearing the framing $-p/q$, and draw $L$ in relation with this diagram. Let $\left <x_{1},\ldots ,x_{n}|\, w_{1},\ldots ,w_{n}\right >$ be the Wirtinger presentation for $\pi _{1}(S^{3}\backslash (L\cup U),*)$, obtained from this diagram. Denote by $m_{1}$ and $l_{1}$ the meridian and longitude of the regular neighbourhood of $S^{3}\backslash U$, written in terms of the generators $x_{1},\ldots ,x_{n}$. Then the presentation for the link group $\pi _{1}(L(p,q)\backslash L,*)$ is given by $$\left <x_{1},\ldots ,x_{n}|\, w_{1},\ldots ,w_{n},m_{1}^{p}\cdot l_{1}^{-q}\right >\;.$$
\end{thm}

The above argument naturally extends to performing rational surgeries on links. 
\begin{thm}[\cite{rolfsen}]
Let $L$ be a link in a closed connected orientable 3-manifold $M$. Represent $M$ as the result of a rational surgery on a $k$-component framed link $L_{0}\subset S^{3}$ and draw the Kirby diagram of $L\cup L_{0}$. Let $\left <x_{1},\ldots ,x_{n}|\, w_{1},\ldots ,w_{n}\right >$ be the Wirtinger presentation of $\pi _{1}(S^{3}\backslash (L\cup L_{0}),*)$. For the $i$-th component of $L_{0}$, let $p_{i}/q_i$ be its framing and denote by $m_{i}$ and $l_{i}$ the meridian and longitude of the regular neighbourhood of its complement in $S^{3}$, written in terms of the generators $x_{1},\ldots ,x_{n}$. Then the group of the link $L$ is given by the presentation $$\pi _{1}(M\backslash L,*)=\left <x_{1},\ldots ,x_{n}|\, w_{1},\ldots ,w_{n},m_{1}^{p_{1}}\cdot l_{1}^{-q_1},\ldots ,m_{k}^{p_{k}}\cdot l_{k}^{-q_k}\right >\;.$$
\end{thm}

Starting with the diagram of $L\subset L(p,q)$, we will now introduce a notation for the generators and relations of the link group to make the calculations easier to follow. 

\begin{notation}
\label{not}
Let $L\subset L(p,q)$ be an oriented link. Represent $L(p,q)$ by its Kirby diagram as the $-p/q$ surgery on an oriented unknot $U$ and draw $L$ in relation with this diagram. By Theorem~\ref{th2}, the presentation of the link group of $L$ is obtained from the presentation of the link group $\pi _{1}(S^{3}\backslash (L\cup U),*)$ by adding one relation. The Wirtinger generators corresponding to the link $L$ will be denoted by $x_{i}$, while the generators corresponding to the unknot $U$ will be denoted by $a_{j}$. 

Let $D$ be the obvious disk in $S^{3}$ that is bounded by $U$. Take a small cylinder $C=D\times [-\epsilon ,\epsilon ]$ that is a regular neighbourhood of $D$ in $S^{3}$. If $L$ is an affine link, then we may assume that $L$ does not cross $C$, and denote by $x_{1},\ldots ,x_{n}$ the generators of $\pi _{1}(L(p,q)\backslash L,*)$, corresponding to $L$, and by $a_{1}$ the generator of $\pi _{1}(L(p,q)\backslash L,*)$, corresponding to $U$. 

If $L$ is not affine, we may assume that $L\cap C$ is a union of $k$ parallel strands $s_{1},\ldots ,s_{k}$. Each of the strands is overcrossed by the unknot $U$, which divides $s_{i}$ into two arcs $\overline{\alpha }_{i}\subset D\times [0,\epsilon]$ and $\overline{\alpha }_{k+i}\subset D\times [-\epsilon ,0]$ (for $i=1,\ldots ,k$). The arc $\overline{\alpha }_{i}$ is a part of the overpass $\alpha _{i}$ in the diagram, and we denote by $x_{i}$ its corresponding generator in $\pi _{1}(L(p,q)\backslash L,*)$ for $i=1,\ldots ,2k$, see Figure~\ref{fig-inter}. It may happen that the arcs $\alpha _{i}$ and $\alpha _{j}$ coincide for two different indices $1\leq i,j\leq 2k$; in this case we add the relation $x_{i}=x_{j}$ to the presentation of $\pi _{1}(L(p,q)\backslash L,*)$. We denote by $a_{1}$ the generator of $\pi _{1}(L(p,q)\backslash L,*)$, corresponding to the overpass of the unknot $U$ which overcrosses the strands $s_{1},\ldots ,s_{k}$. Moreover, we denote by $a_{i}$ the generator of $\pi _{1}(L(p,q)\backslash L,*)$, corresponding to the arc of $U$ which lies between the overcrossings of $s_{k-i+1}$ and $s_{k-i+2}$ with $U$ for $i=2,\ldots ,k$.

\begin{figure}[ht]
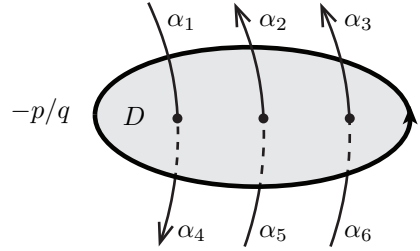

	\centering
	\begin{overpic}[page=12]{images}
		\put(-21,41){$-p/q$}
		\put(12,39){$D$}
		\put(25,68){$\alpha_1$}\put(52,68){$\alpha_2$}\put(77,68){$\alpha_3$}
		\put(28,6){$\alpha_4$}\put(52,6){$\alpha_5$}\put(77,6){$\alpha_6$}
	\end{overpic}
	\caption{Intersections of strands with the disk for $k=3, \epsilon_1 = -1, \epsilon_2=1, \epsilon_3=1$.}\label{fig-inter}
\end{figure}

The orientation of $L$ induces an orientation for each of the strands $s_{1},\ldots ,s_{k}$. For $i=1,\ldots ,k$, let $\epsilon _{i}=1$ if the strand $s_{i}$ is oriented so that the arc $\overline{\alpha }_{i}$ comes after the arc $\overline{\alpha }_{k+i}$, and let $\epsilon _{i}=-1$ otherwise. 
\end{notation}

\begin{cor} 
\label{cor2}
In this notation, the presentation of the link group is given by 
$$ \pi_{1}(L(p,q)\backslash L,*) = \left <x_{1},\ldots ,x_{n},a_{1},\ldots ,a_{k}|\, w_{1},\ldots ,w_{n+k},a_{1}^{p}(x_{1}^{\epsilon _{1}}\ldots x_{k}^{\epsilon _{k}})^{-q}\right >,$$
in particular, if $L$ is affine, we have
$$ \pi_{1}(L(p,q)\backslash L,*) = \left <x_{1},\ldots ,x_{n},a_{1}|\, w_{1},\ldots ,w_{n},a_{1}^{p}\right >.$$
\end{cor}
\begin{proof} It follows directly from Theorem~\ref{th2} and the introduced notation. 
\end{proof}

\begin{lemma} 
\label{lemma6} Any link $L\subset L(p,q)$ with the diagram, described in Notation~\ref{not}, is equivalent to a link whose algebraic intersection with the disk $D$ is an integer between $0$ and $p-1$. \label{sl-lemma}
\end{lemma}
\begin{proof}
The $(p,-q)$-curve is a trivial knot in $L(p,q)$, since it bounds a disk (the meridional disk of $V_2$). A band connected sum with a trivial knot is well defined and keeps the isotopy type of a link in any 3-manifold intact~\cite{gabr} (see also \cite{gmr, lambro}). Since the band connected sum with the $(p,-q)$-curve increases the intersection number by $\pm p$ (the sign depends on the position of the band, see Figure~\ref{fig-slide}) and can be summed any number of times, it follows that any link $L$ is isotopic to a link $L'$ with intersection number between $0$ and $p-1$. 
\end{proof}

Slide move
\begin{figure}[ht]
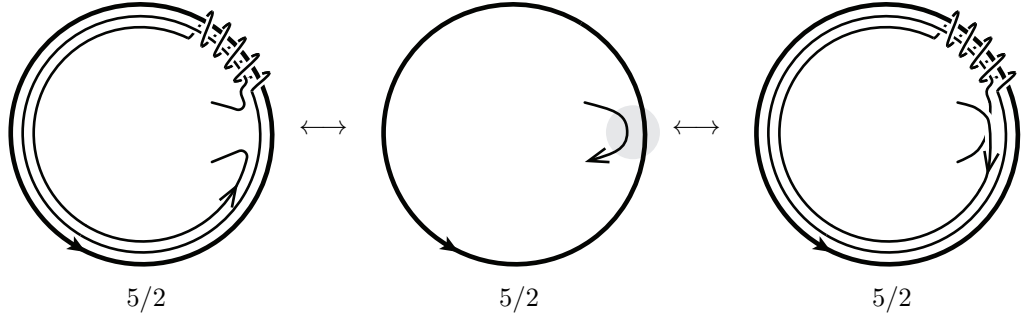

\centering
\begin{overpic}[page=17]{images}
\put(45,-7){$5/2$}
\end{overpic}\raisebox{2cm}{$\;\longleftrightarrow\;$}
\begin{overpic}[page=16]{images}
\put(45,-7){$5/2$}
\end{overpic}\raisebox{2cm}{$\;\longleftrightarrow\;$}
\begin{overpic}[page=27]{images}
\put(45,-7){$5/2$}
\end{overpic}
\vspace{1ex}
\caption{Two possible band connected sums with the meridian in $L(5,2)$.}\label{fig-slide}
\end{figure}

From now on, if not stated otherwise, we will assume that any link $L\subset L(p,q)$ satisfies $0\leq \sum _{i=1}^{k}\epsilon _{i}\leq p-1$. 

To end this Section, we describe the abelianization of the link group of $L$.
\begin{lemma}[\cite{CMM}, Lemma 4]
\label{lemma1}
Let $K\subset L(p,q)$ be an oriented knot. Represent $L(p,q)$ as the $-p/q$ surgery on an unknot $U$, and draw $K$ in relation with this diagram. Let $D$ be the obvious disk in $S^{3}$, bounded by $U$. Denote by $f(K)$ the algebraic intersection number of $K$ and $D$. Then the homology class of $K$ inside $H_{1}(L(p,q))\cong \ZZ _{p}$ is equal to $$[K]=q\cdot f(K)\, \textrm{mod }p\;.$$
\end{lemma}
\begin{proof} The space $L(p,q)$ is obtained from $S^{3}$ by removing the regular neighbourhood $\nu(U)$ of $U$ and gluing in a 2-cell with the boundary $p\cdot l_{1}-q\cdot m_{1}$, where $l_{1}\in \pi _{1}(\partial (\nu(U)))$ is the meridian of $\nu(U)$ and $m_{1}\in \pi _{1}(\partial (\nu(U)))$ is the longitude of $U$. For each positive (negative) point of intersection between $K$ and $D$, there are $q$ positive (negative) points of intersection between $K$ and the 2-cell of $L(p,q)$. By summing up all the signed intersections, the above formula is obtained. 
\end{proof}

\begin{cor}[\cite{CMM}, Corollary 5]
\label{cor3}
Let $L\subset L(p,q)$ be a link with components $L_{1},\ldots ,L_{r}$. Denote by $c_{i}$ the homology class of $L_{i}$ in $H_{1}(L(p,q))\cong \ZZ _{p}$ for $i=1,\ldots r$. The first homology group of the link complement equals $$H_{1}(L(p,q)\backslash L)\cong \ZZ ^{r}\oplus \ZZ _{d}\;,$$ where $d=gcd(c_{1},\ldots ,c_{r},p)$.
\end{cor}
\begin{proof} Represent $L(p,q)$ by the Kirby diagram as the $-p/q$ surgery on an unknot $U$, and draw $L$ in relation with this diagram. Let $D$ be the obvious disk in $S^{3}$, bounded by $U$. For $i=1,\ldots ,r$ denote by $f(L_{i})$ the algebraic intersection number of $L_{i}$ and $D$. We abelianize the link group $\pi _{1}(L(p,q)\backslash L,*)$, whose presentation is given by Corollary~\ref{cor2}. Once the generators commute, then, by the Wirtinger relations $w_{1},\ldots ,w_{n+k}$, all the generators, corresponding to the same link component of $L$, become homologous. Thus, $H_{1}(L(p,q)\backslash L)$ is an abelian group with $r$ free generators $g_{1},\ldots ,g_{r}$, corresponding to the components of the link $L$, and the generator $a_{1}$, corresponding to the unknot $U$. The lens relation, when abelianized, becomes $$p\cdot a_{1}-q\cdot (f(L_{1})g_{1}+\ldots +f(L_{r})g_{r})=0\;.$$ By Lemma~\ref{lemma1}, this is the same as $p\cdot a_{1}-(c_{1}g_{1}+\ldots +c_{r}g_{r})=0$, and it follows that $d=\textrm{gcd }(c_{1},\ldots ,c_{r},p)$. 
\end{proof}

\end{subsection}
\end{section}

\begin{section}{The Alexander-Fox matrix of a link in a lens space}\label{sec-fox}

\begin{subsection}{The construction and definitions}
Given a presentation of the group of a link, one may calculate its Alexander polynomial using Fox calculus. We shortly recall the following construction from~\cite{wada}. Let $$\mathcal{P}=\left <x_{1},\ldots ,x_{n}|\, r_{1},\ldots ,r_{m}\right >$$ be a presentation of a group $G$ and denote by  $H=G/G'$ its abelianization. Let $F=\left <x_{1},\ldots ,x_{n}\right >$ be the corresponding free group. We apply the chain of maps $$\ZZ F\stackrel{\frac{\partial }{\partial x}}{\longrightarrow }\ZZ F\stackrel{\gamma }{\longrightarrow}\ZZ G\stackrel{\alpha }{\longrightarrow}\ZZ H\;,$$ where $\frac{\partial }{\partial x}$ denotes the Fox differential, $\gamma $ is the quotient map by the relations $r_{1},\ldots ,r_{m}$ and $\alpha $ is the abelianization map. The \textit{Alexander-Fox matrix} of the presentation $\mathcal{P}$ is the matrix $A=\left [a_{i,j}\right ]$, where $a_{i,j}=\alpha (\gamma (\frac{\partial r_{i}}{\partial x_{j}}))$ for $i=1,\ldots ,m$ and $j=1,\ldots n$. For $k=1,\ldots ,\min \{m-1,n-1\}$, the \textit{$k$-th elementary ideal} $E_{k}(\mathcal{P})$ is the ideal of $\ZZ H$, generated by the determinants of all the $(n-k)$ minors of $A$. The \textit{first elementary ideal} $E_{1}(\mathcal{P})$ is the ideal of $\ZZ H$, generated by the determinants of all the $(n-1)$ minors of $A$. 

\begin{defn}
Let $L\subset S^{3}$ be a link, and let $E_{k}(\mathcal{P})$ be the $k$-th elementary ideal, obtained from a presentation $\mathcal{P}$ of $\pi _{1}(S^{3}\backslash L,*)$. Then the \textbf{$k$-th link polynomial} $\Delta _{k}(L)$ is the generator of the smallest principal ideal containing $E_{k}(\mathcal{P})$. The \textbf{Alexander polynomial} of $L$, denoted by $\Delta (L)$, is the first link polynomial of $L$. 
\end{defn}

For a classical link $L\subset S^{3}$, the abelianization of $\pi _{1}(S^{3}\backslash L,*)$ is the free abelian group, whose generators correspond to the components of $L$, see Corollary~\ref{cor1}. For a link in $L(p,q)$, the abelianization of its link group may also contain torsion, as we know by Corollary~\ref{cor3}. In this case, we need the notion of a twisted Alexander polynomial. We recall the following from~\cite{CMM}. 

Let $G$ be a group with a finite presentation $\mathcal{P}$ and abelianization $H=G/G'$ and denote $K=H/Tors(H)$. Then every homomorphism $\sigma \colon Tors(H)\to \CC ^{*}=\CC \backslash \{0\}$ determines a twisted Alexander polynomial $\Delta ^{\sigma }(\mathcal{P})$ as follows. Choosing a splitting $H=Tors(H)\times K$, $\sigma $ defines a ring homomorphism $\sigma \colon \ZZ [H]\to \CC [K]$ sending $(f,g)\in Tors(H)\times K$ to $\sigma (f)g$. Thus we apply the chain of maps $$\ZZ F\stackrel{\frac{\partial }{\partial x}}{\longrightarrow }\ZZ F\stackrel{\gamma }{\longrightarrow}\ZZ G\stackrel{\alpha }{\longrightarrow}\ZZ [H]\stackrel{\sigma }{\longrightarrow}\CC [K]$$ and obtain the $\sigma $-twisted Alexander matrix $A^{\sigma }=\left [\sigma (\alpha (\gamma (\frac{\partial r_{i}}{\partial x_{j}})))\right ]$. The twisted Alexander polynomial is then defined by $\Delta ^{\sigma }(\mathcal{P})=\textrm{gcd}(\sigma (E_{1}(\mathcal{P})))$.    

\begin{defn}
Let $L\subset L(p,q)$ be a link in the lens space. For any presentation $\mathcal{P}$ of the link group $\pi _{1}(L(p,q)\backslash L,*)$, we may define the following. 

The \textbf{Alexander polynomial} of $L$, denoted by $\Delta (L)$, is the generator of the smallest principal ideal containing $E_{1}(\mathcal{P})$. 

For any homomorphism $\sigma \colon Tors(H_{1}(L(p,q)\backslash L))\to \CC ^{*}$, the \textbf{$\sigma $-twisted Alexander polynomial} of $L$ is $\Delta ^{\sigma }(L)=\textrm{gcd}(\sigma (E_{1}(\mathcal{P})))$.
\end{defn}

We know from Corollary~\ref{cor3} that the torsion subgroup of $H_{1}(L(p,q)\backslash L)$ is the group $\ZZ _{d}$. Thus, the image of the group homomorphism $\sigma \colon Tors(H_{1}(L(p,q)\backslash L))\to \CC ^{*}$ is contained in the cyclic group, generated by $\zeta $, the $d$-th root of the unity. The $\sigma $-twisted Alexander polynomial $\Delta ^{\sigma }(L)\in \ZZ [\zeta ][K]$ is defined up to multiplication by $\zeta ^{j}g$, with $g\in K$. When the chosen generator $gen$ of $Tors(H_{1}(L(p,q)\backslash L))$ will be clear from the context, we will denote by $\Delta ^{\mu }(L)$ the $\sigma $-twisted Alexander polynomial, for which $\sigma (gen)=\mu \in \CC ^{*}$.  

\end{subsection}

\begin{subsection}{The calculation of the Alexander-Fox matrix}
\label{subs3}
In this subsection we will use the presentation of the link group $\pi _{1}(L(p,q)\backslash L,*)$ to calculate the Fox differentials and obtain the Alexander-Fox matrix $A_{L}$. 

Let $L\subset L(p,q)$ be a link, given by a diagram, described in Notation~\ref{not}. The diagram of $L$ consists of two parts: the first part is contained in the cylinder $C$, and the second part is the diagram in the rest of $S^{3}$. The first part of the diagram determines the Wirtinger relations, corresponding to the crossings between $L$ and $U$, and the lens relation. The second part of the diagram determines the Wirtinger relations, corresponding to the crossings within the link $L$. Using the notation, described in Notation~\ref{not}, we will now write down the relations, determined by the first part of the diagram.

If $L$ is an affine link, then the first part of the diagram consists merely of the unknot $U$, determining a single relation $l\colon a_{1}^{p}$.

If $L$ is not affine, then the first part of the diagram contains $2k$ crossings. By Lemma~\ref{lemma6} we may assume that $0\leq \sum _{i=1}^{k}\epsilon _{i}\leq p-1$. For $i=1,\ldots ,k$, there is a crossing where $U$ overcrosses the strand $s_{i}$, yielding the Wirtinger relation $q_{i}\colon a_{1}x_{k+i}^{\epsilon _{i}}a_{1}^{-1}x_{i}^{-\epsilon _{i}}$. Moreover, for $i=1,\ldots ,k$ there is a crossing where the strand $s_{i}$ overcrosses $U$, yielding the Wirtinger relation $r_{i}\colon x_{i}^{\epsilon _{i}}a_{k-i+1}x_{i}^{-\epsilon _{i}}a_{k-i+2\textrm{(mod $k$)}}^{-1}$. Finally, the lens relation is $l\colon a_{1}^{p}\left (x_{1}^{\epsilon _{1}}\ldots x_{k}^{\epsilon _{k}}\right )^{-q}$. The presentation of the link group is $$\pi _{1}(L(p,q)\backslash L,*)=\left <x_{1},\ldots ,x_{n},a_{1},\ldots ,a_{k}|\, q_{1},\ldots ,q_{k},r_{1},\ldots ,r_{k},w_{1},\ldots ,w_{n-k},l\right >\;,$$ where $w_{1},\ldots ,w_{n-k}$ denote the Wirtinger relations, corresponding to the crossings within the link $L$. 

Now we calculate the Fox differentials of the known relations, apply the quotient map $\gamma $ by the relations and abelianize to obtain the Alexander-Fox matrix. Since the lens relation is essentially different from the other Wirtinger relations, this process is done in two steps, as will be described in the following Proposition. 

\begin{defn} For a link $L\subset L(p,q)$, given by the diagram described in Notation~\ref{not}, we denote $\overline{k}=\sum _{i=1}^{k}\epsilon _{i}$, $p'=\frac{p}{d}$ and $k'=\frac{\overline{k}}{d}$, where $d=\gcd \{p,\overline{k}\}$. 
\end{defn}

\begin{prop} 
\label{prop5}
For a link $L\subset L(p,q)$, let $$\mathcal{P}=\left <x_{1},\ldots ,x_{n},a_{1},\ldots ,a_{k}|\, q_{1},\ldots ,q_{k},r_{1},\ldots ,r_{k},w_{1},\ldots ,w_{n-k},l\right >$$ be the link group presentation, described above. Let $M$ be the matrix of the Fox differentials of $\mathcal{P}$, and denote by $A(x,a)$ the matrix we obtain from $M$ by identifying $x_{i}=x$ for $i=1\ldots ,n$ and $a_{j}=a$ for $j=1,\ldots k$. Then the Alexander-Fox matrix of $L$ is given by $$A_{L}(t)=A(t^{\frac{p}{d}},t^{\frac{q\overline{k}}{d}})=A(t^{p'},t^{qk'})\;,$$ while the $\mu $-twisted Alexander-Fox matrix is given by $A^{\mu }_{L}(t)=A(t^{p'},\nu t^{qk'})$, where $\mu \in \CC ^{*}$ is a $d$-th root of unity and $\nu $ is any complex root of the polynomial $z^{p'}-\mu $.   
\end{prop}
\begin{proof} 
As the Fox differentials of all the relations in the given presentation are calculated, we apply on them the quotient map $\gamma $ by the relations, followed by the abelianization $\alpha $.  
When applying the homomorphism $\alpha \circ \gamma $ with respect to the Wirtinger relations, all the generators, corresponding to the same component of the link $L\cup U$, become identified. We thus identify $\alpha (\gamma (a_{1}))=\ldots =\alpha (\gamma (a_{k}))=a$. Since we will calculate the one-variable Alexander polynomial of $L$, we  also identify $\alpha (\gamma (x_{i}))=x$ for $i=1\ldots ,n$. By applying $\alpha \circ \gamma $ with respect to all relations except the lens relation, we obtain the two-variable matrix $A(x,a)$, with $x$ corresponding to the link $L$ and $a$ corresponding to the unknot $U$. 

The Alexander-Fox matrix of $L$ is calculated from $A(x,a)$ by applying $\alpha \circ \gamma $ with respect to the lens relation and thus identifying $$a^{p}=\alpha (\gamma (a_{1}))^{p}=\alpha (\gamma (a_{1}^{p}))=\alpha (\gamma ((x_{1}^{\epsilon _{1}}\ldots x_{k}^{\epsilon _{k}})^{q}))=x^{q\overline{k}}\;.$$ Therefore, the Alexander-Fox matrix $A_{L}$ of the link $L$ is obtained from $A(x,a)$ by the substitution $A_{L}(t)=A(t^{\frac{p}{d}},t^{\frac{q\overline{k}}{d}})=A(t^{p'},t^{qk'})$. 

By Corollary~\ref{cor3}, the torsion of the abelianized link group is the group $\ZZ _{d}$. If there is nontrivial torsion, the lens relation becomes  $(a^{p'}x^{-qk'})^{d}=1$ and the homomorphism $\sigma \colon \ZZ [H]\to \CC [K]$ sends the torsion generator $a^{p'}x^{-qk'}$ to $\mu $. It follows that the $\mu $-twisted Alexander-Fox matrix $A^{\mu }_{L}$ is obtained from $A(x,a)$ by the substitution $A^{\mu }_{L}(t)=A(t^{p'},\nu t^{qk'})$.
\end{proof}

\begin{rem} 
\label{rem1} Observe that $\overline{k}$ is actually the algebraic intersection number of $L$ and the disk $D$. If $L\subset L(p,q)$ is a link with components $L_{1},\ldots ,L_{r}$, we have by Lemma~\ref{lemma1} $[L]=\sum _{i=1}^{r}[L_{i}]=(q\overline{k})(\textrm{mod }p)\in H_{1}(L(p,q))$.
Since $p$ and $q$ are coprime, it follows that $d=\gcd \{p,\overline{k}\}=\gcd \{p,q\overline{k}\}=\gcd \{p,[L]\}$ and thus the number $p'$ may be more invariantly defined as $p'=\frac{p}{\gcd \{p,[L]\}}$.
\end{rem}
We calculate the matrix $A(x,a)=$
\begin{xalignat}{1}
\label{Axa}
& \bordermatrix{~ & w_{1} & \ldots & w_{n-k} & q_{k} & \ldots & q_{1} & r_{k} & \ldots & r_{1} & l \cr
x_{n} & \quad & \quad & \quad & \quad & \quad & \quad & \quad & \quad & \quad & \quad \cr
\vdots & \quad & B_{1} & \quad & \quad & 0 & \quad & \quad & 0 & \quad & 0 \cr
x_{2k+1} & \quad & \quad & \quad & \quad & \quad & \quad & \quad & \quad & \quad & \quad \cr
x_{2k} & \quad & \quad & \quad & a\phi _{k} & \quad & \quad & \quad & \quad & \quad & \quad \cr 
\vdots & \quad & B_{3} & \quad & \quad & \ddots & \quad & \quad & 0 & \quad & 0 \cr
x_{k+1} & \quad & \quad & \quad & \quad & \quad & a\phi _{1} & \quad & \quad & \quad & \quad \cr
x_{k} & \quad & \quad & \quad & -\phi _{k} & \quad & \quad & (1-a)\phi _{k} & \quad & \quad & \beta _{k} \cr
\vdots & \quad & B_{2} & \quad & \quad & \ddots & \quad & \quad & \ddots & \quad & \vdots \cr
x_{1} & \quad & \quad & \quad & \quad & \quad & -\phi _{1} & \quad & \quad & (1-a)\phi _{1} & \beta _{1} \cr
a_{1} & \quad & \quad & \quad & (1-x)\phi _{k} & \ldots & (1-x)\phi _{1} & x^{\epsilon _{k}} & \quad & -1 & \theta \cr
\vdots & \quad & 0 & \quad & \quad & 0 & \quad & -1 & \ddots & \quad & 0 \cr
a_{k} & \quad & \quad & \quad & \quad & \quad & \quad & \quad & -1 & x^{\epsilon _{1}} & \quad \cr } 
\end{xalignat}
where the missing entries are all zero submatrices and $\theta =1+a+\ldots +a^{p-1}$, 
\begin{xalignat*}{1}
\beta _{i}=-(1+x^{\overline{k}}+\ldots +x^{(q-1)\overline{k}})x^{\sum _{j=1}^{i-1}\epsilon _{j}}\phi _{i}
\end{xalignat*} and
\begin{xalignat*}{1}
\phi _{i}=\begin{cases}
1 &  \textrm{ if }\epsilon _{i}=1, \\
-x^{-1}&\textrm{ if }\epsilon _{i}=-1.
\end{cases}
\end{xalignat*} for $i=1,\ldots ,k$. The $-1$ entries in the right lower part of \ref{Axa} are situated in row $a_{(k-i+2)\textrm{mod} k}$ and column $r_{i}$ for $i=1,\ldots ,k-1$. The entries $B_{1},B_{2},B_{3}$ represent the Fox differentials of the Wirtinger relations $w_{1},\ldots ,w_{n-k}$ of the crossings within the link $L$. 

\end{subsection}
\end{section}

\begin{section}{The main results}\label{sec-main}
In this Section, we describe a relation between the Alexander polynomials of a link in the lens space and of its classical counterpart in $S^{3}$. Let $L\subset L(p,q)$ be a link, given by a diagram, described in Notation~\ref{not}. Denote by $L'\subset S^{3}$ the classical link we obtain from $L$ by ignoring the surgery along the unknot $U$. We will study the relation between the Alexander polynomial $\Delta (L)$ of $L\subset L(p,q)$ and the Alexander polynomials of the classical links $\Delta (L')$ and $\Delta (L'\cup U)$. 
Firstly, we observe the case of an affine link in the lens space.

\begin{cor}[\cite{CMM}, Proposition 7]
\label{cor0}
Let $L\subset L(p,q)$ be an affine link. Then for its twisted Alexander polynomials we have $\Delta ^{1}(L)=p\cdot \Delta  (L')$ and $\Delta ^{\mu }(L)=0$ for $\mu \neq 1$.  
\end{cor}
\begin{proof} By Corollary~\ref{cor2}, the link group of an affine link has a presentation $$\pi _{1}(L(p,q)\backslash L,*)=\left <x_{1},\ldots ,x_{n},a_{1}|\, w_{1},\ldots ,w_{n},a_{1}^{p}\right > \;.$$ Denoting by $A_{L}$ and $A_{L'}$ the Alexander-Fox matrices for $L$ and $L'$ respectively, $A_{L}$ equals $A_{L'}$ with one additional row (corresponding to $a_{1}$) and column (corresponding to the lens relation $l\colon a_{1}^{p}$). From the lens relation it follows that the abelianization of $\pi _{1}(L(p,q)\backslash L,*)$ has the torsion subgroup $\ZZ _{p}$. If $\sigma \colon \ZZ _{p}\to \CC ^{*}$ is the homomorphism, which takes the generator $a_{1}$ to $\mu \neq 1$, we have $\sigma (\alpha (\gamma (\frac{\partial l}{\partial a_{1}})))=\frac{\mu ^{p}-1}{\mu -1}=0$.

For the homomorphism $\sigma$ which takes the generator $a_{1}$ to 1, we have $\sigma (\alpha (\gamma (\frac{\partial l}{\partial a_{1}})))=p$, and this is the only nonzero entry in the last row (and column) of the matrix $A_{L}$. For the minors of $A_{L}$, we observe the following. Since the Wirtinger presentation of a classical link group has deficiency one, we have $\det A_{L}^{n+1,n+1}=\det A_{L'}=0$. Moreover, if we remove from $A_{L}$ either any row and the last column or the last row and any column, then the remaining matrix is singular. It follows that
\begin{xalignat*}{1}
\det A_{L}^{i,j}=\left \{ 
\begin{array}{lr}
p\cdot \det A_{L'}^{i,j} & \textrm{ if }1\leq i,j\leq n,\\
0 & \textrm{ if $i=n+1$ or $j=n+1$},
\end{array}
\right.
\end{xalignat*} and we conclude $$\gcd \{\det A_{L}^{i,j}|\, 1\leq i,j\leq n+1\}=p\cdot \gcd \{\det A_{L'}^{i,j}|\, 1\leq i,j\leq n\}\;.$$ 
\end{proof}

Now we explore the case of a link $L\subset L(p,q)$ which is not affine. In this case, $L$ is nontrivially linked with the unknot $U$. We observe that the Alexander polynomial of $L\subset L(p,q)$ is related to the Alexander polynomial of the classical link $L'\cup U\subset S^{3}$:

\begin{prop} 
\label{prop0} Let $\Delta (L)$ be the Alexander polynomial of  $L\subset L(p,q)$. Denote by $\Delta _{i}(L'\cup U)$ the $i$-th two-variable polynomial of the classical link $L'\cup U\subset S^{3}$ (the variables corresponding to $L'$ and $U$ respectively). Then $\Delta (L)(t)$ divides $\Delta _{1}(L'\cup U)(t^{p'},t^{qk'})$ and is divisible by $\Delta _{2}(L'\cup U)(t^{p'},t^{qk'})$.  
\end{prop}
\begin{proof}
The presentation 
$$\pi _{1}(L(p,q)\backslash L,*)=\left <x_{1},\ldots ,x_{n},a_{1},\ldots ,a_{k}|\, q_{1},\ldots ,q_{k},r_{1},\ldots ,r_{k},w_{1},\ldots ,w_{n-k},l\right >\;,$$ described in Subsection~\ref{subs3}, is obtained from the presentation for $\pi _{1}(S^{3}\backslash (L'\cup U),*)$ by adding the lens relation, see Theorem~\ref{th2}. Applying $\alpha \circ \gamma $ with respect to all relations except the lens relation means identifying all the variables $x_{i}$ and all the variables $a_{i}$ in the Fox differentials of this presentation. We obtain the matrix $A(x,a)$, given in Proposition~\ref{prop5}, and the Alexander-Fox matrix of $L\subset L(p,q)$ is given by $A_{L}(t)=A(t^{p'},t^{qk'})$. By deleting the last column of $A(x,a)$ (belonging to the lens relation), we obtain the two-variable Alexander-Fox matrix of the link $L'\cup U\subset S^{3}$.  
\begin{xalignat*}{1}
& A_{L}(t)=\bordermatrix{~ & w_{1} & \ldots & w_{n-k} & q_{k} & \ldots & q_{1} & r_{k} & \ldots & r_{1} & l \cr
x_{n} & \quad & \quad & \quad & \quad & \quad & \quad & \quad & \quad &\quad & 0 \cr
\vdots & \quad & \quad & \quad & \quad & \quad & \quad & \quad & \quad & \quad & \vdots \cr
x_{k+1} & \quad & \quad & \quad & \quad & \quad & \quad & \quad & \quad & \quad & 0 \cr
x_{k} & \quad & \quad & \quad & \quad & \quad & \quad & \quad & \quad & \quad & \beta _{k} \cr
\vdots & \quad & \quad & \quad & \quad & A_{L'\cup U}(t^{p'},t^{qk'}) & \quad & \quad & \quad & \quad & \vdots \cr
x_{1} & \quad & \quad & \quad & \quad & \quad & \quad & \quad & \quad & \quad & \beta _{1}\cr
a_{1} & \quad & \quad & \quad & \quad & \quad & \quad & \quad & \quad & \quad & \theta \cr
\vdots & \quad & \quad & \quad & \quad & \quad & \quad &  \quad & \quad & \quad & 0 \cr
a_{k} & \quad & \quad & \quad & \quad & \quad & \quad & \quad & \quad & \quad & 0 \cr } 
\end{xalignat*}
where $\theta =\frac{\partial l}{\partial a_{1}}(t^{p'},t^{qk'})$ and $\beta _{i}=\frac{\partial l}{\partial x_{i}}(t^{p'},t^{qk'})$ for $i=1,\ldots ,k$. Denote by $A^{i,(j,j')}$  the matrix we get from $A$ by deleting the $i$-th row, the $j$-th and the $j'$-th column, and denote by $B=A_{L'\cup U}(t^{p'},t^{qk'})$. The Alexander polynomial of $L$ may be written as 
\begin{xalignat*}{1}
& \Delta (L)=\gcd \{\det A_{L}^{i,(j,j')}|\, 1\leq i\leq n+k,\, 1\leq j<j'\leq n+k+1\}\\
& =\gcd \{\det B^{i,j},\theta \det B^{(i,n+1),(j,j')}+\sum _{r=1}^{k}(-1)^{r}\beta _{r}\det B^{(i,n-r+1),(j,j')}|1\leq i,j,j'\leq n+k\}\;.
\end{xalignat*}
Since $\Delta _{1}(L'\cup U)(t^{p'},t^{qk'})=\gcd \{\det B^{i,j}  \mid 1\leq i,j\leq n+k\}$ and $\Delta _{2}(L'\cup U)(t^{p'},t^{qk'})=\gcd \{\det B^{(i,i'),(j,j')} \mid 1\leq i,i',j,j'\leq n+k\}$, it follows that $\Delta _{2}(L'\cup U)(t^{p'},t^{qk'})$ divides $\Delta (L)$ and $\Delta (L)$ divides $\Delta _{1}(L'\cup U)(t^{p'},t^{qk'})$. 
\end{proof}

As we have seen, the group of a link $L\subset L(p,q)$ is in a close relationship with the group of the classical link $L'\cup U\subset S^{3}$. Based on this relationship, Proposition~\ref{prop0} approximates the Alexander polynomial $\Delta (L)$ with the link polynomials of $L'\cup U$. Let us describe the relation between the Alexander polynomials $\Delta (L)$ and $\Delta (L'\cup U)$ more precisely.  

\begin{defn} Let $p$ and $q$ be positive coprime integers. Denote by $\lambda _{p,q}$ the rational function, given by $$\lambda _{p,q}(u)=\frac{u^{pq}-1}{(u^{p}-1)(u^{q}-1)}\;.$$ 
\end{defn}

\begin{lemma}
\label{lemma5} For any positive coprime integers $p$ and $q$ we have $\lambda _{p,q}(u)=\frac{\lambda _{1}(p,q)(u)}{u-1}$, where $\lambda _{1}(p,q)$ is a polynomial not divisible by $(u-1)$.
\end{lemma}
\begin{proof} 
Since $\gcd(p,q)=1$, every root of the denominator is also a root of the numerator. The root $1$, however, is a double root of the denominator and a single root of the numerator.
\end{proof}

We need another classical result, that was obtained in \cite[Theorem, page 61]{torres}.
\begin{lemma}\cite{torres}\label{lemmaT} Let $L$ be a $\mu $-component link in $S^{3}$, where $\mu \geq 2$. Let $A_{L}(t_{1},\ldots ,t_{\mu })$ be an Alexander matrix for $L$, where $t_{i}$ denotes the variable, corresponding to component $L_i$ for $i=1,\ldots ,\mu $. Denote by $A_{L}^{j}$ the minor of $A_{L}$, obtained by deleting the row/column belonging to the generator $x_{j}\in \pi _{1}(S^{3}\backslash L)$. Then there exists a polynomial $\Delta \in \ZZ [t_{1}^{\pm 1},\ldots ,t_{\mu }^{\pm 1}]$ such that  $$A_{L}^{j}=\pm (x_{j}-1)\Delta $$ for $j=1,\ldots ,\mu $. 
\end{lemma}

\begin{thm} 
\label{th4} Let $L$ be a link in $L(p,q)$, which intersects the disk $D$ in $k$ transverse intersection points so that $\overline{k}=\sum _{i=1}^{k}\epsilon _{i}\neq 0$. Let $p'=\frac{p}{d}$ and $k'=\frac{\overline{k}}{d}$ where $d=\gcd \{p,\overline{k}\}$. Then the Alexander polynomial of $L$ and the two-variable Alexander polynomial of the classical link $L'\cup U$ are related by 
$$\Delta (L)(t)=\frac{t-1}{t^{k'}-1}\, \Delta (L'\cup U)(t^{p'},t^{qk'})\;.$$

\end{thm}
\begin{proof} Represent $L$ by a diagram, described in Notation~\ref{not}. Then the group of $L$ has the presentation $$\pi _{1}(L(p,q)\backslash L,*)=\left <x_{1},\ldots ,x_{n},a_{1},\ldots ,a_{k}|\, q_{1},\ldots ,q_{k},r_{1},\ldots ,r_{k},w_{1},\ldots ,w_{n-k},l\right >\;,$$ described in Subsection~\ref{subs3}. By calculating the Fox differentials of this presentation and then identifying all the $x_{i}$-variables and all the $a_{j}$-variables, we obtain the matrix $A(x,a)$. In Proposition~\ref{prop5} we have shown that the Alexander-Fox matrix of $L$ is calculated from $A(x,a)$ by the substitution $A_{L}(t)=A(t^{p'},t^{qk'})$. 

For $i=1,\ldots ,k$, we denote: $\theta =1+a+\ldots +a^{p-1}$, 
\begin{xalignat*}{1}
\beta _{i}=-(1+x^{\overline{k}}+\ldots +x^{(q-1)\overline{k}})x^{\sum _{j=1}^{i-1}\epsilon _{j}}\phi _{i}
\end{xalignat*} and
\begin{xalignat*}{1}
\phi _{i}=\left \{ 
\begin{array}{lr}
1 & \textrm{ if }\epsilon _{i}=1,\\
-x^{-1} & \textrm{ if }\epsilon _{i}=-1\;.
\end{array}
\right.
\end{xalignat*} and then calculate the matrix $A(x,a)=$
\begin{xalignat}{1}
\label{Axa1}
& \bordermatrix{~ & w_{1} & \ldots & w_{n-k} & q_{k} & \ldots & q_{1} & r_{k} & \ldots & r_{1} & l \cr
x_{n} & \quad & \quad & \quad & \quad & \quad & \quad & \quad & \quad & \quad & \quad \cr
\vdots & \quad & B_{1} & \quad & \quad & 0 & \quad & \quad & 0 & \quad & 0 \cr
x_{2k+1} & \quad & \quad & \quad & \quad & \quad & \quad & \quad & \quad & \quad & \quad \cr
x_{2k} & \quad & \quad & \quad & a\phi _{k} & \quad & \quad & \quad & \quad & \quad & \quad \cr 
\vdots & \quad & B_{3} & \quad & \quad & \ddots & \quad & \quad & 0 & \quad & 0 \cr
x_{k+1} & \quad & \quad & \quad & \quad & \quad & a\phi _{1} & \quad & \quad & \quad & \quad \cr
x_{k} & \quad & \quad & \quad & -\phi _{k} & \quad & \quad & (1-a)\phi _{k} & \quad & \quad & \beta _{k} \cr
\vdots & \quad & B_{2} & \quad & \quad & \ddots & \quad & \quad & \ddots & \quad & \vdots \cr
x_{1} & \quad & \quad & \quad & \quad & \quad & -\phi _{1} & \quad & \quad & (1-a)\phi _{1} & \beta _{1} \cr
a_{1} & \quad & \quad & \quad & (1-x)\phi _{k} & \ldots & (1-x)\phi _{1} & x^{\epsilon _{k}} & \quad & -1 & \theta \cr
\vdots & \quad & 0 & \quad & \quad & 0 & \quad & -1 & \ddots & \quad & 0 \cr
a_{k} & \quad & \quad & \quad & \quad & \quad & \quad & \quad & -1 & x^{\epsilon _{1}} & \quad \cr } 
\end{xalignat}
If we erase the last column of $A$, corresponding to the lens relation, we obtain the (two variable) Alexander matrix of the classical link $L'\cup U$. Since $A$ has $n+k$ rows and $n+k+1$ columns, its columns are linearly dependent. Therefore, the last column may be expressed as a linear combination of the previous columns. The last column may be written as a linear combination of the columns $r_{1},\ldots ,r_{k}$ as follows: 
\begin{xalignat}{1}
\label{lamb}
l=\lambda _{p,q}(t^{k'})\left ((x^{\sum _{i=1}^{k-1}\epsilon _{i}})r_{k}+(x^{\sum _{i=1}^{k-2}\epsilon _{i}})r_{k-1}+\ldots +(x^{\epsilon _{1}})r_{2}+r_{1}\right )\;.
\end{xalignat} We may use this linear combination when calculating the determinants of the minors of $A_{L}$. The Alexander polynomial of $L$ is given by $$\Delta (L)=\gcd \{\det A_{L}^{i,(j,j')}|\, 1\leq i\leq n+k,1\leq j<j'\leq n+k+1\}\;.$$
We observe the minors $A_{L}^{i,(j,j')}$ and compare them with the minors of $A_{L'\cup U}$. Denote $B(t)=A_{L'\cup U}(t^{p'},t^{qk'})$.  For $j'=n+k+1$ we have $\det A_{L}^{i,(j,n+k+1)}=\det B^{i,j}$. 

The minor $A_{L}^{i,(j,j')}$ for $j'\leq n+k$ is obtained from the minor $B^{i,j}$ by changing the $j'$-th column to $l$ and then subtracting from $l$ all possible terms in the linear combination \eqref{lamb} (we may subtract the terms whose corresponding columns are different from $j$ and $j'$). 

For $1\leq j<j'\leq n$ we have $\det A_{L}^{i,(j,j')}=0$, since the last $(k+1)$ columns of this minor are linearly dependent. 

For $j\leq n$ and $j'=n+r$ we have $\det A_{L}^{i,(j,n+r)}=\lambda _{p,q}(t^{k'})x^{\sum _{i=1}^{k-r}\epsilon _{i}}\det B^{i,j}$ if $1\leq r\leq k$. 

For the remaining case when $n+1\leq j<j'\leq n+r$, we argue as follows. If $i\leq n$, then the rows $a_{2},\ldots ,a_{k}$ are linearly dependent since each of them has at most $(k-2)$ nontrivial entries. If, on the other hand, $i>n$, then we have $$\det A_{L}^{i,(j,j')}=\lambda _{p,q}(t^{k'})\left (x^{\sum _{i=1}^{k-(j-n)}\epsilon _{i}}\det B^{i,j'}+x^{\sum _{i=1}^{k-(j'-n)}\epsilon _{i}}\det B^{i,j}\right )\;.$$ Thus, we may write 
\begin{xalignat*}{1}
& \det A_{L}^{i,(j,j')}=\\
& \left \{ 
\begin{array}{lr}
0, & \textrm{ $1\leq j<j'\leq n$,}\\
\det B^{i,j}, & \textrm{ $j'=n+k+1$,}\\
\lambda _{p,q}(t^{k'})x^{\sum _{i=1}^{k-r}\epsilon _{i}}\det B^{i,j}, & \textrm{ $j\leq n$, $j'=n+r$, $1\leq r\leq k$,}\\
0, & \textrm{ $i\leq n$, $n+1\leq j<j'\leq n+k$,}\\
\lambda _{p,q}(t^{k'})\left (x^{\sum _{i=1}^{k-(j-n)}\epsilon _{i}}\det B^{i,j'}+x^{\sum _{i=1}^{k-(j'-n)}\epsilon _{i}}\det B^{i,j}\right ), & \textrm{ $i>n$, $n+1\leq j<j'\leq n+k$.}
\end{array} 
\right.
\end{xalignat*}

By Lemma \ref{lemmaT} it follows that 
\begin{xalignat*}{1}
& \textrm{gcd}\{\det B^{i,j}|\, 1\leq i,j\leq n+k\}\\
& =\textrm{gcd}\{t^{p'}-1,t^{qk'}-1\}\Delta (L'\cup U)(t^{p'},t^{qk'})=(t-1)\Delta (L'\cup U)(t^{p'},t^{qk'})\;.
\end{xalignat*}
Since the Alexander polynomial is only defined up to the multiplication by a power of $t$, it follows by Lemma~\ref{lemma5} that 
\begin{xalignat*}{1}
& \Delta (L)(t)=\gcd \{\det A_{L}^{i,(j,j')}|\, 1\leq i\leq n+k,1\leq j<j'\leq n+k+1\}\\
& =\frac{t-1}{t^{k'}-1}\, \Delta (L'\cup U)(t^{p'},t^{qk'}).
\end{xalignat*}
\end{proof}

\subsection{Alexander polynomial of  null-homologous links in $L(p,q)$}
We obtain the following corollaries.

\begin{cor} 
\label{cor7} Let $L$ be a link in $L(p,q)$, which intersects the disk $D$ in $k$ transverse intersection points, so that $\overline{k}=\sum _{i=1}^{k}\epsilon _{i}=0$. Then the Alexander polynomial of $L$ and the two-variable Alexander polynomial of the classical link $L'\cup U$ are related by 
$$\Delta (L)(t)=\Delta (L'\cup U)(t,t^{q})\;.$$
\end{cor}
\begin{proof} Represent $L$ by a diagram, described in Notation~\ref{not}. By taking the band connected sum with the $(p,-q)$-curve (recall Lemma \ref{sl-lemma}), we may change $L$ to an equivalent link whose algebraic intersection with $D$ equals $\overline{k}=\sum _{i=1}^{k}\epsilon _{i}=p$. Then we have $\gcd \{p,\overline{k}\}=p$, $p'=\frac{p}{d}=1$ and $k'=\frac{\overline{k}}{d}=1$. We proceed as in the proof of Theorem \ref{th4}. 
\end{proof}

\begin{cor} 
\label{cor5} Let $L\subset L(p,q)$ be a link, which intersects the disk $D$ in $k$ transverse intersection points, so that $\overline{k}=0$. Then the $\mu $-twisted Alexander polynomial of $L$ and the two-variable Alexander polynomial of the classical link $L'\cup U$ are related by $$\Delta ^{\mu }(L)(t)=\frac{\Delta (L'\cup U)(t,\mu )}{\mu -1}\;,$$ where $\mu \in \CC $ is any $p$-th complex root of unity, different from 1.   
\end{cor}
\begin{proof} We use the same method as in the proof of Theorem~\ref{th4}. In this case, $d=\gcd \{p,\overline{k}\}=p$ and therefore $p'=1$ and $k'=0$. The $\mu $-twisted Alexander-Fox matrix of $L$ is calculated from $A(x,a)$ by the substitution $A^{\mu }_{L}(t)=A(t,\mu )$ (see Proposition~\ref{prop5}). For $i=1,\ldots ,k$ we calculate: $\theta =1+a+\ldots +a^{p-1}=\frac{\mu ^{p}-1}{\mu -1}=0$ and
\begin{xalignat*}{1}
\beta _{i}=-qx^{\sum _{j=1}^{i-1}\epsilon _{j}}\phi _{i}\;.
\end{xalignat*} The matrix $A(x,a)$ looks like \eqref{Axa1}. If we erase the last column of $A(x,a)$, corresponding to the lens relation, we obtain the (two variable) Alexander matrix of the classical link $L'\cup U$. The last column may be written as a linear combination of the columns $r_{1},\ldots ,r_{k}$ as follows: 
\begin{xalignat}{1}
\label{lambda}
l=\frac{q}{1-\mu }\left ((x^{\sum _{i=1}^{k-1}\epsilon _{i}})r_{k}+(x^{\sum _{i=1}^{k-2}\epsilon _{i}})r_{k-1}+\ldots +(x^{\epsilon _{1}})r_{2}+r_{1}\right )\;.
\end{xalignat}
Denote $B(t)=A_{L'\cup U}(t,\mu )$. Using an analogous reasoning as in the proof of Theorem~\ref{th4}, we calculate
\begin{xalignat*}{1}
& \det A_{L}^{i,(j,j')}=\\
& \left \{ 
\begin{array}{lr}
0, & \textrm{ $1\leq j<j'\leq n$,}\\
\det B^{i,j}, & \textrm{ $j'=n+k+1$,}\\
\frac{q}{1-\mu }x^{\sum _{i=1}^{k-r}\epsilon _{i}}\det B^{i,j}, & \textrm{ $j\leq n$, $j'=n+r$, $1\leq r\leq k$,}\\
0, & \textrm{ $i\leq n$, $n+1\leq j<j'\leq n+k$.}\\
\frac{q}{1-\mu }\left (x^{\sum _{i=1}^{k-(j-n)}\epsilon _{i}}\det B^{i,j'}+x^{\sum _{i=1}^{k-(j'-n)}\epsilon _{i}}\det B^{i,j}\right ), & \textrm{ $i>n$, $n+1\leq j<j'\leq n+k$.}
\end{array} 
\right.
\end{xalignat*}
Since the Alexander polynomial is only defined up to multiplication by a power of $t$, it follows that 
\begin{xalignat*}{1}
& \Delta ^{\mu }(L)(t)=\frac{\Delta (L'\cup U)(t,\mu )}{\mu -1}\;.
\end{xalignat*}
\end{proof}

\begin{cor} Denote by $-L$ the link $L\subset L(p,q)$ with the opposite orientation. Then the Alexander polynomials of $L$ and $-L$ are connected by $$\Delta (-L)(t)=\Delta (L)(t^{-1})\;.$$
\end{cor}
\begin{proof} By Theorem~\ref{th4}, the Alexander polynomial of $L$ is given by $\Delta (L)(t)=\frac{t-1}{t^{k'}-1}\Delta _{L'\cup U}(t^{p'},t^{qk'})$. When changing the orientation of $L$, every sign $\epsilon _{i}$ switches to $-\epsilon _{i}$, thus $\overline{k}=\sum _{i=1}^{k}\epsilon _{i}$ becomes $-\overline{k}$ and consequently $k'$ becomes $-k'$. The classical link $L'$, corresponding to $L$, also changes orientation, and for the classical links we know that $\Delta (-L')(t)=\Delta (L')(t^{-1})$. Therefore, by Theorem~\ref{th4} we have $$\Delta (-L)(t)=\frac{t-1}{t^{-k'}-1}\, \Delta (-L'\cup U)(t^{p'},t^{-qk'})=\frac{t-1}{t^{-k'}-1}\Delta (L'\cup U)(t^{-p'},t^{-qk'})=\Delta (L)(t^{-1})\;.$$
\end{proof}

\subsection{Alexander polynomial of a family of unlinks in $L(p,q)$}

Now we observe a family of unlinks $\{L_{k,r}\}_{k\in \mathbb{N},0\leq r\leq k}$ in $L(p,q)$. Any link $L\subset L(p,q)$ may be obtained by combining ("multiply" connect summing) the classical link $L'\subset S^{3}$ with a suitable link $L_{k,r}$. We calculate the Alexander polynomial $\Delta (L_{k,r})$ for any $k\in \mathbb{N},0\leq r\leq k$. In the following proposition, we begin with the subfamily $\{L_{k}\}_{k\in \mathbb{N}}$, where $L_{k}=L_{k,0}$.     

\begin{prop}
\label{prop1} Let $k$ be a positive integer. Denote by $L_{k}\subset L(p,q)$ the unlink, consisting of $k$ unknots, each of them intersecting the disk $D$ transversely in a single positive point of intersection. The Alexander polynomial of $L_{k}$ is given by 
\begin{xalignat*}{1}
& \Delta (L_{k})(t)=\frac{(t^{qk'}-1)^{k-1}(t-1)}{t^{k'}-1}\;.
\end{xalignat*}
\end{prop}
\begin{proof} We will calculate the Alexander polynomial of the classical link $L_{k}'\cup U\subset S^{3}$ and then use Theorem~\ref{th4}. Denote by $x_{i}$ the generator of $\pi _{1}(S^{3}\backslash (L_{k}'\cup U),*)$, corresponding to the $i$-th unknot of the unlink $L_ {k}$ for $i=1,\ldots ,k$. Denote by $a_{1},\ldots ,a_{k}$ the generators, corresponding to the meridian $U$, so that the Wirtinger relations of the crossings between the meridian and the unknots are $q_{i}\colon a_{1}x_{i}a_{1}^{-1}x_{i}^{-1}$ and $r_{i}\colon x_{i}a_{k-i+1}x_{i}^{-1}a_{(k-i+2)\textrm{ mod}k}$ for $i=1,\ldots ,k$. From the link group presentation $$\pi _{1}(S^{3}\backslash (L_{k}'\cup U),*)=\left <x_{1},\ldots ,x_{k},a_{1},\ldots ,a_{k}|\, q_{1},\ldots ,q_{k},r_{1},\ldots ,r_{k}\right >$$ we derive the Alexander matrix $A_{L_{k}'\cup U}$. For $k\geq 2$, we have 
\begin{xalignat*}{1}
& A_{L_{k}'\cup U}(x,a)=\bordermatrix{~ & q_{k} & \ldots & q_{1} & r_{k} & \ldots & r_{1}\cr
x_{k} & (a-1) & \quad & \quad & (1-a) & \quad & \quad \cr
\vdots & \quad & \ddots & \quad & \quad & \ddots & \quad \cr
x_{1} & \quad & \quad & (a-1) & \quad & \quad & (1-a) \cr
a_{1} & (1-x) & \ldots & (1-x) & x & \quad & -1 \cr 
\vdots & \quad & 0 & \quad & -1 & \ddots & \quad \cr
a_{k} & \quad & \quad & \quad & \quad & -1 & x \cr } 
\end{xalignat*} Observe that this matrix is obtained from the matrix $A(x,a)$ in \eqref{Axa} if $n=k$, by adding the row $x_{k+i}$ to $x_{i}$ for $i=1,\ldots ,k$ and then deleting the rows $x_{k+1},\ldots ,x_{2k}$. 
 
The Alexander polynomial is then calculated as $$\Delta (L_{k}'\cup U)(x,a)=\gcd \{\det A_{L_{k}'\cup U}(x,a)^{i,j}|\, 1\leq i,j\leq 2k\}\;.$$ where $A^{i,j}$ denotes the minor, obtained by deleting the $i$-th row and the $j$-th column of a matrix $A$. Denote by $S(i,j)$ the set of all bijections $$\pi \colon \{1,\ldots ,\widehat{i},\ldots ,2k\}\to \{1,\ldots ,\widehat{j},\ldots ,2k\}\;,$$ and let $m_{i,j}$ be the $(i,j)$-th entry of the matrix $A_{L_{k}'\cup U}$. Then we have $$\det A_{L_{k}'\cup U}^{i,j}=\sum _{\pi \in S(i,j)}m_{1,\pi (1)}m_{2,\pi (2)}\ldots m_{2k,\pi (2k)}\;.$$ From the matrix, we may observe that for every $\pi \in S(i,j)$, the following holds. 
\begin{xalignat*}{1}
& m_{1,\pi (1)}\ldots m_{k,\pi (k)}=\left \{
\begin{array}{lr}
0 & \textrm{ or }\\
\pm (a-1)^{k-1} & \textrm{ if }i\leq k,\\
\pm (a-1)^{k} & \textrm{ if }i\geq k+1\;.
\end{array}
\right.
\end{xalignat*}

\begin{xalignat*}{1}
& m_{k+1,\pi (k+1)}\ldots m_{2k,\pi (2k)}=\left \{
\begin{array}{lr}
0 & \textrm{ or }\\
(1-x)(\ldots ) \textrm{ or }x^{k}\textrm{ or }-1 & \textrm{ if }i\leq k\\
\textrm{ something } & \textrm{ if }i\geq k+1
\end{array}
\right.
\end{xalignat*}

If $i\geq k+1$, then $\det A_{L_{k}'\cup U}^{i,j}$ is divisible by $(a-1)^{k}$. If $i\leq k$ and $j\geq k+1$, then $\det A_{L_{k}'\cup U}^{i,j}=(a-1)^{k-1}(1-x)(-1)^{j-k-1}x^{2k-j}$. If $i,j\leq k$, then we may expand the determinant along the $(k+1)$-st row to obtain 
\begin{xalignat*}{1}
& \det A_{L_{k}'\cup U}^{i,j}=(a-1)^{k-1}\left (\pm (x^{k}-1)+(1-x)\sum _{r=1}^{k}C_{r}\right )\;,
\end{xalignat*} where the factors $C_{r}$ denote sums of the terms $m_{k+2,\pi (k+2)}\ldots m_{2k,\pi (2k)}$. 

It follows that the determinant $\det A_{L_{k}'\cup U}^{i,j}$ for $1\leq i,j\leq 2k$ is always divisible by $(a-1)^{k-1}\gcd \left \{a-1,x-1\right \}$. Since $\det A_{L_{k}'\cup U}^{1,2k}(x,a)=(a-1)^{k-1}(1-x)(-1)^{k-1}$ and $\det A_{L_{k}'\cup U}^{k+1,2k}(x,a)=(a-1)^{k}(-1)^{k-1}$, we conclude that $$\Delta (L_{k}'\cup U)(x,a)=(a-1)^{k-1}\gcd \left \{a-1,x-1\right \}\;.$$ By Theorem~\ref{th4} it follows 
\begin{xalignat*}{1}
& \Delta (L_{k})(t)=\frac{(t^{qk'}-1)^{k-1}}{t^{k'}-1}\gcd \left \{t^{qk'}-1,t^{p'}-1\right \}=\frac{(t^{qk'}-1)^{k-1}(t-1)}{t^{k'}-1}\;.
\end{xalignat*} Here we have used the fact that the numbers $p'$ and $qk'$ are coprime, and thus 1 is the only common root of the polynomials $t^{qk'}-1$ and $t^{p'}-1$.  

 For $k=1$ we have 
\begin{xalignat*}{1}
A_{L_{1}'\cup U}(x,a)=\bordermatrix{~ & q_{1} & r_{1} \cr
x_{1} & a-1 & 1-a \cr
a_{1} & 1-x & x-1 \cr }
\end{xalignat*} and $\Delta (L_{1}'\cup U)(t^{p'},t^{q})=\gcd \{t^{q}-1,t^{p'}-1\}=t-1$ and it follows from Theorem~\ref{th4} that $\Delta (L_{1})=1$. 
\end{proof}

\subsection{Relation between the Alexander polynomials of a link in $L(p,q)$ and its corresponding link in $S^{3}$}

\begin{thm} 
\label{th}
Let $L\subset L(p,q)$ be a link which intersects the disk $D$ transversely in a single point of intersection. Denote by $L'\subset S^{3}$ the link with the same diagram as $L$, which we get by ignoring the surgery along the unknot $U$. Then the Alexander polynomials of $L$ and $L'$ are related by $$\Delta (L)(t)=\Delta (L')(t^{p})\;.$$
\end{thm}
\begin{proof} In this case, $k=1$ and $H_{1}(L(p,q)\backslash L)$ contains no torsion. Take the strand of $L$ which intersects the disk $D$, and cut it on each side of $D$ to represent $L$ as the connected sum $L'\#L_{1}$ ($L_{1}$ is the unlink with one component, defined in Proposition~\ref{prop1}). We use the result~\cite[Proposition 8]{CMM}, which says the Alexander polynomial of a connected sum of two links in $L(p,q)$, one of them a local link, equals the product of both Alexander polynomials. Let
$\pi _{1}(S^{3}\backslash L',*)=\left <y_{1},\ldots ,y_{n}|\, w_{1},\ldots ,w_{n}\right >$  be the Wirtinger presentation of the group of $L'$ and  
$\pi _{1}(L(p,q)\backslash L_{1},*)=\left <x_{1},a_{1}|\, q_{1},r_{1},l\right >$  be the presentation of the group of $L_{1}$ as defined in the proof of Proposition~\ref{prop1}. Then $$\pi _{1}(L(p,q)\backslash L,*)=\left <y_{1},\ldots y_{n},x_{1},a_{1}|\, w_{1},\ldots ,w_{n},q_{1},r_{1},l,x_{1}=y_{1}\right >$$ is the presentation of the group of $L$. The Alexander matrix of $L$ looks like 
\begin{xalignat*}{1}
A_{L}(t)=\begin{pmatrix} 
\quad & \quad & \quad & \quad & \quad & \quad & -1 \cr
\quad & A_{L'}(t^{p}) & \quad & \quad & 0 & \quad & 0 \cr
\quad & \quad & \quad & \quad & \quad & \quad & \vdots \cr
\quad & \quad & \quad & \quad & \quad & \quad & 0 \cr  
\quad & 0 & \quad & \quad & A_{L_{1}}(t) & \quad & 1 \cr
\quad & \quad & \quad & \quad & \quad & \quad & 0 \cr 
\end{pmatrix}\;,
\end{xalignat*}
where $A_{L'}(t)$ and $A_{L_{1}}(t)$ are the Alexander matrices of $L'$ and $L_{1}$ respectively. In the Alexander matrix for $L$, all the generators, corresponding to the overpasses of $L$ (these are $y_{1},\ldots ,y_{n}$ and $x_{1}$) are identified and sent to $t^{p}$, while the generator $a_{1}$ is sent to $t^{q}$ (see Proposition~\ref{prop5}). The Alexander polynomial of $L$ is calculated by 
$$\Delta (L)(t)=\gcd \{  \det A_{L}(t)^{i,(j,j',j'')}|\, 1\leq i\leq n+2, 1\leq j<j'<j''\leq n+4\}\;.$$ For the indices $1\leq i\leq n+2$ and $1\leq j<j'<j''\leq n+4$ we have
\begin{xalignat*}{1}
& \det A_{L}(t)^{i,(j,j',j'')}=\\
& \left \{
\begin{array}{lr}
\det A_{L'}(t^{p})^{i,j}\cdot \det A_{L_{1}}(t)^{1,(j'-n,j''-n)} & \textrm{ if $i\leq n$, $j\leq n$ and $n+1\leq j',j''\leq n+3$},\\
-\det A_{L'}(t^{p})^{1,j}\cdot \det A_{L_{1}}(t)^{i,(j'-n,j''-n)} & \textrm{ if $i\geq n+1$, $j\leq n$ and $n+1\leq j',j''\leq n+3$},\\
0 & \textrm{ otherwise. }
\end{array}
\right.
\end{xalignat*}
Denoting by $d_{i}(A)$ the greatest common divisor of all $i$-minors of a matrix $A$, we have $d_{n+1}(A_{L}(t))=d_{n-1}(A_{L'}(t^{p}))\cdot d_{1}(A_{L_{1}}(t))$ and consequently $\Delta (L)(t)=\Delta (L')(t^{p})\cdot \Delta (L_{1})(t)$. We have shown in Proposition~\ref{prop1} that $\Delta (L_{1})=1$ and thus $\Delta (L)(t)=\Delta (L')(t^{p})$. 
\end{proof}

\begin{cor} Let $L\subset L(p,q)$ be a link, intersecting the disk $D$ transversely in a single point of intersection. Then for the corresponding link $L'\subset S^{3}$ we have $\Delta (L'\cup U)(t^{p},t^{q})=\Delta (L')(t^{p})$.
\end{cor}
\begin{proof} In this case $k=1$. By Theorem~\ref{th} we have $\Delta (L)(t)=\Delta (L')(t^{p})$, and by Theorem~\ref{th4} we conclude $\Delta (L'\cup U)(t^{p},t^{q})=\Delta (L')(t^{p})$. 
\end{proof}

\end{section}



\subsection{Examples}

\begin{ex} \label{example1}
Denote by $K_{i}\subset L(p,q)$ the trefoil knots in Figure~\ref{fig-tref}. Note that $K_i$  intersects the disk $D$ in $i$ transverse intersection points for $i=0,1,2$. The corresponding classical trefoil knot $K'\subset S^{3}$ has $\Delta (K')(t)=t^{2}-t+1$ and $\Delta _{2}(K')(t)=1$. By Corollary~\ref{cor0}, the (twisted) Alexander polynomial of $K_{0}$ equals $$\Delta ^{1}(K_{0})(t)=p(t^{2}-t+1)\;.$$ By Theorem~\ref{th}, the Alexander polynomial of $K_{1}$ equals $\Delta (K_{1})(t)=t^{2p}-t^{p}+1$.  
For $K_{2}\in L(p,q)$, we calculate $\Delta (K_{2}'\cup U)(x,a)=(x^{3}+a)\gcd \{x-1,a-1\}$. It follows from Theorem~\ref{th4} that the Alexander polynomial of $K_{2}$ is given by 
\begin{xalignat*}{1}
& \Delta (K_{2})(t)=\left \{
\begin{array}{lr}
\frac{\Delta (K_{2}'\cup U)(t^{p},t^{2q})}{t^{2}-1}=\frac{(t^{3p}+t^{2q})(t-1)}{t^{2}-1}=\frac{t^{2q}(t^{3p-2q}+1)}{t+1} & \textrm{ if $p$ is odd,}\\
\frac{\Delta (K_{2}'\cup U)(t^{p'},t^{q})}{t-1}=\frac{(t^{3p'}+t^{q})(t-1)}{t-1}=t^{q}(t^{\frac{3p-2q}{2}}+1) & \textrm{ if $p$ is even}\;.
\end{array}
\right.
\end{xalignat*} It might be interesting to check whether a similar formula defines the Alexander polynomial of a general torus knot in $L(p,q)$.  
\end{ex}

\begin{figure}[ht]
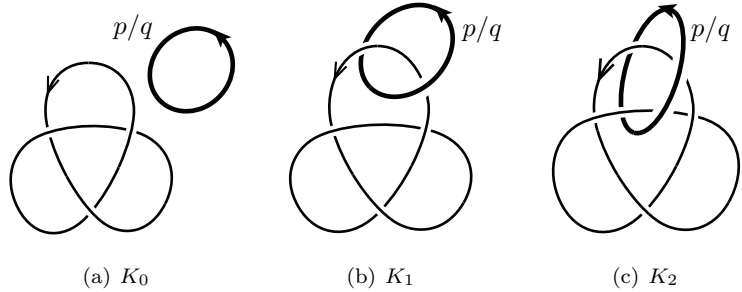

	\centering
	\subfigure[$K_0$]{\begin{overpic}[page=7]{images}
			\put(49,85){$p/q$}
			\label{fig-tref0}
	\end{overpic}}\hspace{0ex}
	\subfigure[$K_1$]{\begin{overpic}[page=8]{images}
				\put(82,84){$p/q$}
				\label{fig-tref1}
	\end{overpic}}
	\subfigure[$K_2$]{\begin{overpic}[page=9]{images}
			\put(68,84){$p/q$}
				\label{fig-tref2}
	\end{overpic}}

			\caption{Diagrams of knots $K_1$, $K_2$, and $K_3$ from Example \ref{example1}.}\label{fig-tref}
\end{figure}




\begin{ex} \label{example2}
The Alexander polynomial is indeed a useful invariant that can detect inequivalent knots that are not distinguishable by other common invariants.
Let $K_1$ and $K_2$ be knots in $L(3,1)$ with diagrams from Figure~\ref{ex2}, which are respectively the knots 
$\overline{1_1}$ and $\overline{5_1}$ from the knot table in~\cite{gabr} and respectively the links $L4a1$ and $L10n42$ from the Thistlethwaite link table with $3/1$ rational surgery performed on the trivial component.
The knots are not distinguished by the Kauffman bracket skein module $S_{2,\infty}$, an invariant that generalizes the Kauffman bracket polynomial (see Table 4 in \cite{gabr} and Appendix D in \cite{gabr2}), since it holds that
$$S_{2,\infty}(K_1) = S_{2,\infty}(K_2) = -A^9+A^4 x+A.$$
Both knots are also homologous, $[K_1]=[K_2]\in H_1(L(p,1)) \cong \ZZ_3$. We can easily see that their complements have the same homology groups: $H_1(L(p,1)\setminus K_1) \cong H_1(L(p,1)\setminus K_1) \cong \ZZ$. 
However, the Alexander polynomial detects that they are indeed different knots:

$$\Delta(K_1)(t) = 1$$
$$\Delta(K_2)(t) = t^{-6} -t^{-5} + t^{-2} - 1 + t^2 - t^5 + t^6.$$

\end{ex}

\begin{figure}[ht]
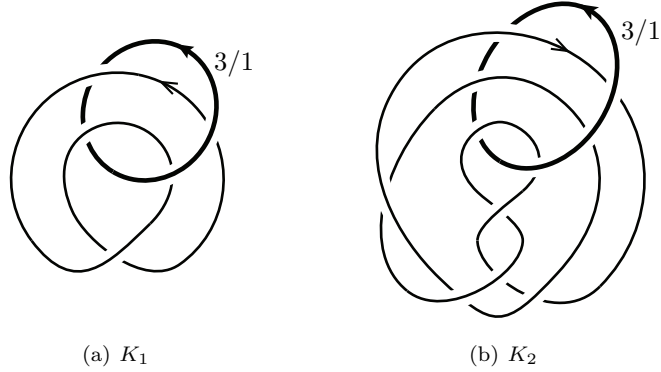

	\centering
	\subfigure[$K_1$]{\begin{overpic}[page=10]{images}
			\put(65,75){$3/1$}
			\label{knot1}
		\end{overpic}}\hspace{7ex}
		\subfigure[$K_2$]{\begin{overpic}[page=11]{images}
				\put(84,84){$3/1$}
				\label{knot2}
			\end{overpic}}
				
				\caption{Two knots in $L(3,1)$ from Example \ref{example2}}\label{ex2}
\end{figure}

\section{A skein relation}
\label{sec:skein}

The relationship between the Alexander polynomial of links in lens spaces and the classical Alexander polynomial obtained in Section \ref{sec-main} may be used to find a skein relation for the normalized version of the Alexander polynomial for links in lens spaces. 

\subsection{The case of lens spaces}

Let $L_+, L_-,$ and $L_0$ be three links in a 3-manifold $M$ that are identical except inside a 3-ball, where they look like those in Figure~\ref{triple}. We refer to $L_+, L_-, L_0$ as a skein triple.

\begin{figure}[ht]
	\centering
	\subfigure[$L_+$]{\includegraphics[page=24]{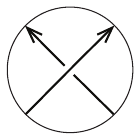}\label{triple1}}\hspace{7ex}
	\subfigure[$L_-$]{\includegraphics[page=25]{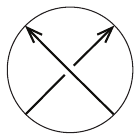}\label{triple2}}\hspace{7ex}
	\subfigure[$L_0$]{\includegraphics[page=26]{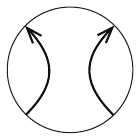}\label{triple3}}
				\caption{The links $L_+,L_-,L_0$ are identical except at one crossing.}\label{triple}
\end{figure}

Conway showed that the classical Alexander polynomial, can be uniquely normalized by taking
$$\nabla(L)(t) = \pm t^{\frac{n}{2}} \Delta(L)(t),$$
for some $n \in \ZZ$, see~\cite{kauff} for more details. The polynomial $\nabla(L)(t)$ is called the normalized Alexander polynomial or the Alexander-Conway polynomial and is without indeterminacies in contrast to the Alexander polynomial which is well-defined up to multiplication by the monomial $\pm t^{\frac{n'}{2}}$, $n' \in \ZZ$.

For the skein triple $L_+$, $L_-$, $L_0 \subset S^3$, the following {\textbf{ skein relation}} holds~\cite{kauff}: $$\nabla(L_{+})-\nabla(L_{-})=(t^{\frac{1}{2}}-t^{-\frac{1}{2}})\nabla(L_{0})\;.$$
 

\begin{defn}
Let $L$ be a link in $L(p,q)$ obtained by $-p/q$ surgery on the oriented unknot $U$ and let $L'$ be the link by ignoring surgery on $U$. Let $\overline{k}$ be the algebraic intersection number of $L$ and the disk $D$ bounding $U$.
We define the normalized Alexander polynomial of $L$ to be

$$
\nabla(L)(t) =
\begin{cases}
\frac{t-1}{t^{k'}-1}\, \nabla(L' \cup U)(t^{p'},t^{qk'}),& \text{if } \bar{k} \neq 0\\
\nabla(L' \cup U)(t,t^q),& \text{if } \bar{k} = 0\\
\end{cases}
$$

where $\nabla(L' \cup U)(t^{p'},t^{qk'})$ and $\nabla(L' \cup U)(t,t^q)$ are the (classical) normalized Alexander polynomials in the single variable $t$ and $p'$, $q$, and $k'$ are defined as in Theorem~\ref{th4}.
\end{defn}
Note that by Theorem~\ref{th4} and Corollary~\ref{cor7}, it holds that $\nabla(L)(t)$ is a normalization of $\Delta(L)(t)$.

\begin{lemma} For a skein triple $L_+, L_-, L_0 \subset L(p,q)$, the homology classes of these links are the same, $[L_+]=[L_-]=[L_0]\in H_{1}(L(p,q))$. \label{lemma:homsum}
\end{lemma}
\begin{proof}
By Remark~\ref{rem1}, the homology class is determined by the algebraic intersection number of $L$ with the disk $D$ bounded by the unknot $U$ on which $-p/q$ surgery is performed.
Without loss of generality we may assume the 3-ball inside which the skein triple differs is disjoint from $D$, and by the fact that the skein triple keeps the orientation of the arcs outside the 3-ball intact, the intersection numbers of $D$ and, respecively, $L_+$, $L_-$, and $L_0$, are the same.
\end{proof}

The following theorem describes the skein relation for the normalized Alexander polynomial in $L(p,q)$.

\begin{thm} 
Let $L_+, L_-, L_0 \subset L(p,q)$ be a skein triple and let $p'=\frac{p}{\gcd \{p,[L]\}}$, where $L$ is either $L_+$, $L_-$, or $L_0$. The following skein relation holds for the normalized Alexander polynomial:
 $$\nabla (L_{+})-\nabla(L_{-})=(t^{\frac{p'}{2}}-t^{-\frac{p'}{2}})\nabla (L_{0})\;.$$ \label{thm:skein1}
 
\end{thm}

\begin{proof} By Lemma~\ref{lemma:homsum}, $p'$ does not depend on the choice of $L$. 
For $\bar k \neq 0$ we obtain

\begin{xalignat*}{1}
& \nabla (L_{+})(t)-\nabla (L_{-})(t)=\frac{t-1}{t^{k'}-1}\left (\nabla (L_{+}'\cup U)(t^{p'},t^{qk'})-\nabla (L_{-}'\cup U)(t^{p'},t^{qk'})\right )\\
& =\frac{t-1}{t^{k'}-1}\left ((t^{p'})^{\frac{1}{2}}-(t^{p'})^{-\frac{1}{2}}\right )\nabla (L_{0}'\cup U)(t^{p'},t^{qk'})=\left (t^{\frac{p'}{2}}-t^{-\frac{p'}{2}}\right )\nabla (L_{0})(t),
\end{xalignat*}
where the first and third equalities hold by definition and the second equality holds from the fact that 
$L_{+}'\cup U$, $L_{-}'\cup U$, $L_{0}'\cup U$ is also a skein triple in $S^3$.

Similarly, for $\bar{k} = 0$ it holds that $p' = 1$ and by Corollary~\ref{cor7}  we have
\begin{xalignat*}{1}
& \nabla (L_{+})(t)-\nabla (L_{-})(t)=\nabla (L_{+}'\cup U)(t,t^{q})-\nabla (L_{-}'\cup U)(t,t^{q})\\
& =\left (t^{\frac{1}{2}}-t^{-\frac{1}{2}}\right )\nabla (L_{0}'\cup U)(t,t^{q})=\left (t^{\frac{p'}{2}}-t^{-\frac{p'}{2}}\right )\nabla (L_{0})(t),
\end{xalignat*}
by the same arguments as before.

\end{proof}

\subsection{Links in other 3-manifolds}

We have obtained a relationship between the Alexander polynomial of a link in $L(p,q)$ and the Alexander polynomial of its classical counterpart in $S^{3}$. It is possible to use the same approach in a more general setting. Let $M$ be a closed, orientable, connected 3-manifold. By the Lickorish-Wallace theorem \cite{lick,wallace}, $M$ may be represented by integral surgery on a link $U_{1}\cup U_{2}\cup \ldots \cup U_{n}$ in $S^{3}$, where $U_{i}$ is an unknot for $i=1,2,\ldots ,n$. Any link in $M$ may be drawn in relation with the corresponding Kirby diagram, containing the surgery link. For a link $L$ in the manifold $M$, denote by $L'$ the corresponding link in $S^{3}$ we obtain by ignoring the surgery along the unknots $U_{1},\ldots ,U_{n}$. The Alexander polynomial of the link $L$ may be expressed by the classical Alexander polynomial of the link $L'\cup U_{1}\cup \ldots \cup U_{n}$, as the following theorem states. 

\begin{thm} \label{th5} Let $M$ be a 3-manifold, obtained by integral surgery on a link $U_{1}\cup U_{2}\cup \ldots \cup U_{n}$, where $U_{i}$ is an unknot for $i=1,2,\ldots ,n$. Then there exist integers $b_{0},b_{1},\ldots ,b_{n}$ and a polynomial $P\in \RR [t] $ such that the Alexander polynomial of any link $L$ in the manifold $M$ may be expressed as $$\Delta _{L}(t)=\frac{\Delta _{L'\cup U_{1}\cup \ldots \cup U_{n}}(t^{b_{0}},t^{b_{1}},\ldots ,t^{b_{n}})}{P(t)}\;.$$
\end{thm}
\begin{proof} We prove the statement of the theorem by induction on $n$ (the number of components of the surgery link). 

For $n=1$, the manifold $M$ is the lens space $L(p,q)$. By Theorem \ref{th4} and Corollary \ref{cor7} we have $\Delta _{L}(t)=\frac{\Delta _{L'\cup U_{1}}(t^{p'},t^{qk'})}{1+t+\ldots +t^{k'-1}}$ if $k'\neq 0$ and $\Delta _{L}(t)=\Delta _{L'\cup U_{1}}(t,t^{q})$ if $k'=0$.

Now suppose the statement of the theorem holds for some $n\in \NN $. Let $M$ be a 3-manifold, obtained by integral surgery on a link $U_{1}\cup U_{2}\cup \ldots \cup U_{n+1}$, where $U_{i}$ is an unknot for $i=1,2,\ldots ,n+1$. Denote by $M'$ the 3-manifold, obtained by the same surgery on the link $U_{1}\cup U_{2}\cup \ldots \cup U_{n}$, but ignoring the surgery on the last component $U_{n+1}$. For a link $L$ in the manifold $M$, let $L''$ be the link in $M'$, corresponding to $L$ (and let $L'$ be the corresponding link in $S^{3}$). By the induction hypothesis, there exist integers $b_{0},b_{1},\ldots ,b_{n}$ and a polynomial $P\in \RR [t]$ such that the Alexander polynomial of $L''$ may be expressed as $$\Delta _{L''}(t)=\frac{\Delta _{L'\cup U_{1}\cup \ldots \cup U_{n}}(t^{b_{0}},t^{b_{1}},\ldots ,t^{b_{n}})}{P(t)}\;.$$ The polynomial $P=P_{n}(t)$ is obtained from the polynomial $P_{1}(t)=1+t+\ldots +t^{k_1-1}$ by a recursive formula $P_{i+1}(t)=P_{i}(t^{p_{i+1}})(1+t+\ldots +t^{k_{i+1}-1})$, where $p_{i}$ and $k_{i}$ are some positive integers for $i=1,\ldots ,n$. Thus, $P(t)$ is a polynomial with real coefficients. 

The manifold $M\backslash L$ is obtained from the manifold $M'\backslash (L''\cup U_{n+1})$ by performing the integral surgery on the component $U_{n+1}$. The presentation of the link group of $L$ in $M$ is obtained from the presentation of the link group of $L''\cup U_{n+1}$ in $M'$ by adjoining the lens relation $$l\colon a^{p}(x_{1}^{\epsilon _{1}}\ldots x_{k}^{\epsilon _{k}})^{-1}\;,$$ where $p$ is the surgery coefficient of the component $U_{n+1}$ and $a$ corresponds to its Wirtinger generator. The other generators $x_{1},\ldots ,x_{k}$ are the Wirtinger generators of those arcs of $L'\cup U_{1}\cup \ldots \cup U_{n}$ that are overcrossing the component $U_{n+1}$. 

We have $$\Delta _{L''\cup U_{n+1}}(x,a)=\frac{\Delta _{L'\cup U_{1}\cup \ldots \cup U_{n}\cup U_{n+1}}(x^{b_{0}},x^{b_{1}},\ldots ,x^{b_{n}},a^{b_{0}})}{P(x)}\;.$$ After abelianization, the term $x_{i}$ in the lens relation is replaced by $x^{b_{0}}$ if it corresponds to a strand of $L'$, and is replaced by $x^{b_{j}}$ if it corresponds to a strand of the unlink $U_{j}$. Thus, after abelianization, the lens relation becomes $a^{p}=x^{r}$ for some integer $r$. Denote $d=\gcd (p,r)$, $p'=\frac{p}{d}$ and $k'=\frac{r}{d}$. 

The Alexander-Fox matrix of $L$ is obtained from the Alexander-Fox matrix $A(x,a)$ of the link $L''\cup U_{n+1}$ (where $x$ corresponds to the link $L''$ and $a$ corresponds to the unlink $U_{n+1}$) by adding the column, corresponding to the lens relation and substituting $x=t^{p'}$, $a=t^{k'}$ (we use the same reasoning as in the proof of Proposition \ref{prop5}). Thus, the Alexander-Fox matrix of $L$ looks like 
\begin{xalignat*}{1}
& A_{L}(t)=\bordermatrix{~ & \quad & \quad & \quad & \quad & \quad & \quad & r_{k} & \ldots & r_{1} & l \cr
x_{m} & \quad & \quad & \quad & \quad & \quad & \quad & \quad & \quad &\quad & 0 \cr
\vdots & \quad & \quad & \quad & \quad & \quad & \quad & \quad & \quad & \quad & \vdots \cr
x_{k+1} & \quad & \quad & \quad & \quad & \quad & \quad & \quad & \quad & \quad & 0 \cr
x_{k} & \quad & \quad & \quad & \quad & \quad & \quad & \quad & \quad & \quad & \beta _{k} \cr
\vdots & \quad & \quad & \quad & \quad & A_{L''\cup U_{n+1}}(t^{p'},t^{k'}) & \quad & \quad & \quad & \quad & \vdots \cr
x_{1} & \quad & \quad & \quad & \quad & \quad & \quad & \quad & \quad & \quad & \beta _{1}\cr
a_{1} & \quad & \quad & \quad & \quad & \quad & \quad & \quad & \quad & \quad & \theta \cr
\vdots & \quad & \quad & \quad & \quad & \quad & \quad &  \quad & \quad & \quad & 0 \cr
a_{k} & \quad & \quad & \quad & \quad & \quad & \quad & \quad & \quad & \quad & 0 \cr } 
\end{xalignat*}
where $\theta =\frac{\partial l}{\partial a}(t^{p'},t^{k'})$ and $\beta _{i}=\frac{\partial l}{\partial x_{i}}(t^{p'},t^{k'})$ for $i=1,\ldots ,k$ and all other  missing submatrices  are zero. The Wirtinger relation $r_{i}\colon x_{i}^{\epsilon _{i}}a_{k-i+1}x_{i}^{-\epsilon _{i}}a_{k-i+2\textrm{(mod $k$)}}$ corresponds to a crossing where the unlink $U_{n+1}$ is overcrossed by $x_{i}$. A calculation of the Fox differentials, that is essentially the same as in the proof of Theorem \ref{th4}, shows that the column $l$ of the Alexander-Fox matrix is a linear combination of the columns $r_{1}, \ldots ,r_{k}$: $$l=\frac{1}{t^{k'}-1}\left (x^{c_{k}}r_{k}+x^{c_{k-1}}r_{k-1}+\ldots +x^{c_{1}}r_{1}\right )$$ for some integers $c_{1},\ldots ,c_{k}$. The Alexander polynomial of $L$ is given by 
$$\Delta _{L}(t)=\gcd \{\det A_{L}(t)^{i,(j_{1},\ldots ,j_{n+1})}|\, 1\leq i\leq m+k,1\leq j_{1}<j_{2}<\ldots <j_{n+1}\leq m+k+n+1\}\;.$$
Denote $B(t)=A_{L''\cup U_{n+1}}(t^{p'},t^{k'})$. By similar reasoning as in the proof of Theorem \ref{th4}, we obtain 
\begin{xalignat*}{1}
& \Delta _{L}(t)=\gcd \{\det B(t)^{i,(j_{1},\ldots ,j_{n})},\, \frac{1}{t^{k'}-1}\det B(t)^{i,(j_{1},\ldots ,j_{n})}|\\
& 1\leq i\leq m+k,1\leq j_{1}<j_{2}<\ldots <j_{n}\leq m+k+n\}\\
& =\frac{t-1}{t^{k'}-1}\, \Delta _{L''\cup U_{n+1}}(t^{p'},t^{k'})
\end{xalignat*} and thus 
$$\Delta _{L}(t)=\frac{\Delta _{L'\cup U_{1}\cup \ldots \cup U_{n}\cup U_{n+1}}(t^{p'b_{0}},t^{p'b_{1}},\ldots ,t^{p'b_{n}},t^{k'b_{0}})}{P(t^{p'})(1+t+\ldots +t^{k'-1})}\;.$$
\end{proof}

Theorem \ref{th5} might be the starting point for a study of the Alexander polynomial of links in other 3-manifolds. Based on our results, one would naturally ask the following question: 
Given a closed, orientable, connected 3-manifold $M$, is there a normalization of the Alexander polynomial of links in $M$ that satisfies a skein relation, and if so, what is this skein relation? 







\section*{Acknowledgements}
The first author was supported by the Slovenian Research Agency grant N1-0083.\\
The second author was supported by the Slovenian Research Agency grants N1-0083, N1-0064, J1-8131, and J1-7025.

\end{document}